%% file: MSS_h12.tex
\documentclass[a4paper,12pt]{amsart}

\usepackage{amssymb}
\usepackage{latexsym}
\usepackage{amsmath}
\usepackage{amscd}
\usepackage{graphicx}
\usepackage{palatino}

\title[An abelian quotient of the symplectic derivation Lie algebra of the free Lie algebra]
{An abelian quotient of the symplectic derivation Lie algebra of the free Lie algebra}

\author{Shigeyuki Morita}
\address{Graduate School of Mathematical Sciences, 
The University of Tokyo, 
3-8-1 Komaba, 
Meguro-ku, Tokyo, 153-8914, Japan}
\email{morita@ms.u-tokyo.ac.jp}

\author{Takuya Sakasai}
\address{Graduate School of Mathematical Sciences, 
The University of Tokyo, 
3-8-1 Komaba, 
Meguro-ku, Tokyo, 153-8914, Japan}
\email{sakasai@ms.u-tokyo.ac.jp}

\author{Masaaki Suzuki}
\address{Department of Frontier Media Science, 
Meiji University, 
4-21-1 Nakano, Nakano-ku, Tokyo, 164-8525, Japan}
\email{macky@fms.meiji.ac.jp}


\newtheorem{thm}{Theorem}[section]

\theoremstyle{definition}

\newtheorem{remark}[thm]{Remark}

\newcommand{\hg}{\mathfrak{h}_{g,1}}

\newcommand{\hinf}{\mathfrak{h}_{\infty,1}}

\newcommand{\Symp}[1]{\mathrm{Sp}(2g,\mathbb{#1})}
\newcommand{\symp}[1]{\mathfrak{sp}(2g,\mathbb{#1})}

\newcommand{\Ker}{\mathop{\mathrm{Ker}}\nolimits}
\newcommand{\Hom}{\mathop{\mathrm{Hom}}\nolimits}
\renewcommand{\Im}{\mathop{\mathrm{Im}}\nolimits}

\newcommand{\Out}{\mathop{\mathrm{Out}}\nolimits}
\newcommand{\Aut}{\mathop{\mathrm{Aut}}\nolimits}
\newcommand{\Q}{\mathbb{Q}}

\newcounter{fig}
\begin{document}

\maketitle

\begin{abstract}
We construct an abelian quotient of the symplectic derivation Lie algebra $\hg$ of the free Lie algebra 
generated by the fundamental representation of $\Symp{Q}$. 
More specifically, we show that the weight $12$ part of the abelianization of $\hg$ is 
$1$-dimensional for $g \ge 8$. The computation is done with the aid of computers. 
\end{abstract}

\section{Introduction}\label{sec:intro}

Let $H$ be the fundamental representation over $\mathbb{Q}$ 
of the symplectic group $\mathrm{Sp} (2g,\mathbb{Q})$. 
Topologically, the vector space $H$ is the first rational homology group 
of a compact connected oriented surface of genus $g$ with one boundary component. 
In this paper, we concern with the symplectic derivation Lie algebra $\hg$ 
of the free Lie algebra $\mathcal{L} (H)$ generated by $H$. 
The Lie algebra $\hinf:=\displaystyle\lim_{g \to \infty} \hg$ obtained by the stabilization 
is just the {\it Lie case} of the three infinite-dimensional graded Lie algebras considered 
by Kontsevich in \cite{kontsevich1, kontsevich2}. In these papers, he proved that 
the Lie algebra homology of $\hinf$ is isomorphic to the free graded commutative algebra
generated by the stable homology of the Lie algebra $\mathfrak{sp}(2h,\Q)$ of $\mathrm{Sp} (2h,\mathbb{Q})$ 
and the totality of the cohomology of 
the outer automorphism groups $\mathrm{Out}\,F_n$ of free groups of rank $n \ge 2$. 

In general, computing (co)homology groups of $\mathrm{Out}\,F_n$ has been 
a difficult problem. More specifically, although the theory of outer spaces due to 
Culler and Vogtmann \cite{cuv} gives a $(2n-3)$-dimensional finite cell 
complex which computes the rational (co)homology groups of $\mathrm{Out}\,F_n$ for any fixed $n$, 
the number of cells grows very fast comparing with increasing $n$. 
$H_\ast (\mathrm{Out}\,F_n; \mathbb{Q})$ were 
determined by Hatcher and Vogtmann \cite{hv} for $n \le 4$ (together with general computational results), 
Gerlits \cite{gerlits} for $n =5$, and Ohashi \cite{ohashi} for $n=6$. 
On the other hand, in our previous papers \cite{mss3, mss4}, 
we computed the integral Euler characteristics 
\[e(\mathrm{Out}\, F_{n})=
\displaystyle\sum_{i=0}^{2n-3} (-1)^i \dim 
\left(H_i (\mathrm{Out}\, F_n; \mathbb{Q})\right)\] 
of $\mathrm{Out}\, F_n$ up to $n \le 11$: 

\begin{center}
\begin{tabular}{|c||c|c|c|c|c|c|c|c|c|c|}
\hline
$n$ & $2$ & $3$ & $4$ & $5$ & $6$ & $7$ & $8$ & $9$ & $10$ & $11$\\
\hline
$e(\mathrm{Out}\, F_n)$ & 
$1$ & $1$ & $2$ & $1$ & $2$ & $1$ & $1$ & $-21$ & $-124$ & $-1202$\\
\hline
\end{tabular}
\end{center}
This result shows the existence of many non-trivial {\it odd}-dimensional rational
(co)homology classes of $\mathrm{Out}\, F_{n}$. 
Note that almost all of the above results were obtained with the aid of computers. 

Recently, Bartholdi \cite{bar} computed the rational homology of $\mathrm{Out}\, F_7$. 
The computation was also aided by computers. 
After a big calculation, he obtained 
\[H_i (\mathrm{Out}\, F_7;\mathbb{Q})\cong \begin{cases}
\mathbb{Q} & (i=0,8,11)\\
0 & (\text{otherwise})\end{cases}.\]
In this result, $H_{11} (\mathrm{Out}\, F_7;\mathbb{Q})\cong \mathbb{Q}$  is remarkable 
because it is the first non-trivial {\it odd} and {\it $($virtually$)$ top} rational homology group 
which is explicitly described. 
Here recall that the complex constructed by Culler and Vogtmann is  $(2n-3)$-dimensional, which is known to give 
the virtual cohomological dimension of $\mathrm{Out}\, F_n$. 
It is a sharp contrast with the fact that the virtually top rational homology groups vanish 
for related groups such as $\mathrm{GL}(n,\mathbb{Z})$ and $\mathrm{SL}(n,\mathbb{Z})$ 
(see Lee-Szczarba \cite{ls}), 
and the mapping class groups of surfaces of genus $g \ge 2$ with at most one puncture 
(see our paper \cite{mss1} or Church-Farb-Putman \cite{cfp}, 
see also Conant-Kassabov-Vogtmann \cite{ckv}).  

By applying Kontsevich's theorem to the above fact that $H_{11} (\mathrm{Out}\, F_7;\mathbb{Q})\cong \mathbb{Q}$, 
we have $H_1 (\hinf)_{12} \cong \mathbb{Q}$, where $H_1 (\hinf)_{12}$ is the weight $12$ part of 
the abelianization of the graded Lie algebra $\hinf$ (see Section \ref{sec:derivation} for details). 
The purpose of the present paper is to give an explicit description of 
this isomorphism purely in terms of the Lie algebra $\hinf$. Consequently, we give an alternative proof of 
$H_{11} (\mathrm{Out}\, F_7;\mathbb{Q})\cong \mathbb{Q}$ by a different method. 
Our computation is also aided by computers. In general, it is desirable for a computer-aided result 
to be checked by multiple methods. 

Our computation of $H_1 (\hinf)_{12}$ is given by calculating 
$H_1 (\hg)_{12}$ for sufficiently large $g$. We construct an $\Symp{Q}$-invariant 
cocycle 
\[C: \hg (12) \longrightarrow \mathbb{Q}\]
and show that it is non-trivial for all $g \ge 2$. 
This means that $H_1 (\hg)_{12}$ is non-trivial even for $g$ in the unstable range and 
we finally show that $C$ gives an isomorphism $H_1 (\hg)_{12} \cong \mathbb{Q}$ for $g \ge 8$. 
Note that the non-triviality of $H_1 (\hg)_{12}$ in the unstable range 
was unknown, 
and in the paper \cite{ms} by Gw\'ena\"el Massuyeau and the second named author, 
we give a topological application of this fact. 
The full description for $C$ is put in Appendix.  
The authors are trying to understand the meaning of our cocycle $C$ with a hope to generalize it 
in higher weights (see Section \ref{sec:rem}), 
although it seems difficult because the cocycle is very big and complicated.

{\it Acknowledgement} 
The authors would like to thank Gw\'ena\"el Massuyeau for 
helpful comments on the first draft of this paper. 
They are also grateful to Karen Vogtmann and 
James Conant for fruitful discussions. In particular, 
James Conant informed that he had checked the result 
$H_{11} (\Out F_7;\mathbb{Q}) \cong \mathbb{Q}$ by a different method from 
Bartholdi's and the authors'. 
The authors were partially supported by KAKENHI 
(No.~15H03618, No.~15H03619, No.~16H03931, and No.~16K05159), 
Japan Society for the Promotion of Science, Japan.

\section{The symplectic derivation Lie algebra of the free Lie algebra}\label{sec:derivation}

The fundamental representation $H$ of $\mathrm{Sp} (2g,\mathbb{Q})$ is 
$2g$-dimensional and has a natural non-degenerate anti-symmetric bilinear form 
\[\mu: H \otimes H \longrightarrow \mathbb{Q}.\]
We identify $H$ with the dual space $H^\ast$ by $\mu$. 
Fix a symplectic basis $\{a_1, b_1, \ldots, a_g, b_g\}$ of $H$ with respect to $\mu$. 
That is, they satisfy
\[\mu(a_i, a_j) = \mu(b_i, b_j)=0, \qquad \mu(a_i,b_j)=\delta_{i,j}\]
for any $1 \le i,j \le g$. 

The free Lie algebra $\mathcal{L}(H)$ generated by $H$ has a natural graded 
Lie algebra structure. Let  
\[\mathcal{L}(H) = \bigoplus_{i=1}^\infty \mathcal{L}_i (H)\]
denote the corresponding direct sum decomposition with $\mathcal{L}_i (H)$ 
the degree $i$ part. We now consider the space $\hg$ of all derivations of $\mathcal{L}(H)$ 
that annihilate the symplectic element 
$\omega_0:= \sum_{i=1}^g [a_i, b_i] \in \mathcal{L}_2 (H)$. 
That is, an endomorphism $D$ of $\mathcal{L}(H)$ is in $\hg$ if it satisfies 
the {\it Leibniz rule}
\[D([X,Y]) =[D(X), Y] +[X,D(Y)]\]
for all $X, Y \in \mathcal{L}(H)$ and $D(\omega_0)=0$.  
$\hg$ is a Lie subalgebra of the Lie algebra of all endomorphisms of $\mathcal{L}(H)$. 
Let $\hg (k)$ be the degree $k$ homogeneous part of $\hg$. 
Since the Leibniz rule says that each element of $\hg$ is characterized by 
its action on $\mathcal{L}_1 (H)=H$, we have  
\[\hg (k)=\{D \in \hg \mid D(H) \subset \mathcal{L}_{k+1} (H)\}.\]
In other words, $\hg (k)$ is regarded as a subspace of 
\[\Hom (H, \mathcal{L}_{k+1} (H))= H^\ast \otimes \mathcal{L}_{k+1} (H) = H \otimes \mathcal{L}_{k+1} (H).\]
From this point of view, it is known that 
\[\hg (k)=\Ker \left( H \otimes \mathcal{L}_{k+1} (H) 
 \xrightarrow{[ \cdot, \cdot]} \mathcal{L}_{k+2} (H)\right)\]
and in particular, $\hg (0)$ coincides with the Lie algebra $\symp{Q}$ of $\Symp{Q}$. 
It is easy to check that the direct sum decomposition 
\[\hg = \bigoplus_{k=0}^\infty \hg (k)\]
gives a graded Lie algebra structure. 
We call $\hg$ the {\it symplectic derivation Lie algebra of $\mathcal{L}(H)$}.
The symplectic group $\mathrm{Sp}(2g,\Q)$ acts naturally on each $\hg (k)$.  
This action is the restriction of the diagonal action on $H \otimes \mathcal{L}_{k+1} (H)$, and 
therefore it is compatible with the stabilization $\hg \hookrightarrow \mathfrak{h}_{g+1, 1}$. 

In this paper, we concern with the abelianization $H_1 (\hg)=\hg / [\hg, \hg]$ 
of the Lie algebra $\hg$ as well as $H_1 (\hinf)=\hinf/[\hinf, \hinf]$ after taking 
the direct limit with respect to $g$. 
The grading of $\hg$ gives a direct sum decomposition  
\[H_1 (\hg) = \bigoplus_{w=0}^\infty H_1 (\hg)_w\]
with 
\[H_1 (\hg)_w := \hg (w) \Big/ \sum_{i=0}^{w} [ \hg (i),  \hg (w-i)]\]
called the {\it weight $w$ part}. 
Now we recall the theorem of Kontsevich mentioned in Introduction. 
For simplicity, we only mention the part related to $H_1 (\hinf)$ and 
refer to the original papers \cite{kontsevich1, kontsevich2} and 
a paper by Conant and Vogtmann \cite{cv0} for details. 
In these papers, it is shown that there exists an isomorphism 
\[H_1 (\hinf)_{2n} \cong H^{2n-1} (\Out F_{n+1};\mathbb{Q})\]
for any integer $n \ge 1$. 
Here $\Out F_{n+1}$ is the outer automorphism group of the free group of rank $(n+1)$ and 
$H^{2n-1} (\Out F_{n+1};\mathbb{Q})$ denotes the $(2n-1)$-st rational 
cohomology group of $\Out F_{n+1}$. 
By a technical reason, we rewrite $H_1 (\hinf)_{2n}$ as follows. 
Let
\[\hg^+=\bigoplus_{k=1}^\infty \hg(k)\]
be the ideal of the {\it positive} degree part. 
The spaces $H_1 (\hg)_w$ and $H_1 (\hg^+)_w$ are both $\mathrm{Sp}(2g,\Q)$-modules. 
As shown in the above cited papers (see also \cite[Section 2]{mss3}), 
we have 
\[H_1 (\hg)_w \cong H_1 (\hg^+)_w^{\mathrm{Sp}} 
=\hg (w)^{\mathrm{Sp}} \Big/ \left(\sum_{i=1}^{w-1} [ \hg (i),  \hg (w-i)]\right)^{\mathrm{Sp}}\]
for any $w \ge 1$. 
Here, for an $\mathrm{Sp}(2g,\Q)$-module $V$, we denote by $V^{\mathrm{Sp}}$ 
the invariant subspace for the $\mathrm{Sp}(2g,\Q)$-action and by 
$V_{\mathrm{Sp}}$ the coinvariant quotient of $V$. 
The general theory of $\mathrm{Sp}(2g,\Q)$-representations says that 
for a finite-dimensional representation $V$, 
$V^{\mathrm{Sp}}$ and $V_{\mathrm{Sp}}$ are isomorphic. Hence 
$\hg^{\mathrm{Sp}} \cong (\hg)_{\mathrm{Sp}}$. It is also known that  
$H_\ast (\hg^+)_w^{\mathrm{Sp}} \cong \big(H_\ast (\hg^+)_w\big)_{\mathrm{Sp}}$ stabilize when 
$g$ becomes large. In particular, $H_1 (\hinf)_{2n}$ is {\it finite}-dimensional. 

\section{Main results}

The main result of this paper is to derive 
$H_1 (\hinf^+)_{12}^{\mathrm{Sp}} \cong \Q$ without using theorems of Kontsevich and Bartholdi. 
That is, we prove it directly in $\mathfrak{h}_{\infty,1}^+$. In the proof, we 
construct an explicit linear map (cocycle) $C$ which gives the above isomorphism. 
More precisely, we have:

\begin{thm}
There exists an $\Symp{Q}$-invariant linear map  $C: \mathfrak{h}_{g,1} (12) \to \Q$ satisfying that 
\begin{itemize}
\item $C$ is non-trivial for any $g \ge 2$, 

\item the restriction of $C$ to 
$\displaystyle\sum_{i=1}^{11} [ \mathfrak{h}_{g,1} (i),  \mathfrak{h}_{g,1} (12-i)]$ is trivial.
\end{itemize}
That is, the cocycle $C$ gives a surjection 
$\widetilde{C}: H_1 (\mathfrak{h}_{g,1}^+ )_{12}^{\mathrm{Sp}} \cong 
\big(H_1 (\mathfrak{h}_{g,1}^+ )_{12}\big)_{\mathrm{Sp}} 
\twoheadrightarrow \Q$ for every $g \ge 2$. Moreover $\widetilde{C}$ is an isomorphism for $g \ge 8$.  
\end{thm}
This theorem gives an alternative proof of $H^{11} (\Out F_7;\mathbb{Q}) \cong \mathbb{Q}$. 

\begin{remark}
Since $\mathfrak{h}_{1,1} (12)^{\mathrm{Sp}}=0$ 
as mentioned in \cite{mss5}, we have 
$H_1 (\mathfrak{h}_{1,1}^+)_{12}^{\mathrm{Sp}}=0$. 
Therefore our bound of genus for the non-triviality of $H_1 (\mathfrak{h}_{g,1}^+)_{12}^{\mathrm{Sp}}$ 
is best possible. 
\end{remark}

\section{Method for computation}

In this section, we explain how we prove the main theorem with the aid of computers. 
Our computation for 
$H_1 \big(\mathfrak{h}_{g,1}^+)_{12}^{\mathrm{Sp}}$ proceeds in the following way: 
\begin{enumerate}
\item Find a coordinate system of $\mathfrak{h}_{g,1} (12)^{\mathrm{Sp}} \cong \mathbb{Q}^{650}$.

\item Compute the bracket map 
\[[ \,\cdot\, ,\,\cdot \,]: 
\left(\displaystyle\bigoplus_{i=1}^{11} \big(\mathfrak{h}_{g,1} (i) \otimes \mathfrak{h}_{g,1} (12-i)\big)\right)^{\mathrm{Sp}} 
\longrightarrow \mathfrak{h}_{g,1} (12)^{\mathrm{Sp}}\] 
and see that the image 
includes a $649$-dimensional subspace $W$. 

\item Find a linear map $C: \hg (12)^\mathrm{Sp} \twoheadrightarrow \mathbb{Q}$ 
which annihilates $W$.

\item Check that $C$ is trivial on the image of the bracket map.
\end{enumerate}

\subsection{Coordinate system of $\mathfrak{h}_{g,1} (12)^{\mathrm{Sp}}$ }\label{subsec:detector}
\hskip 12pt We begin by finding out a coordinate system of $\mathfrak{h}_{g,1} (12)^{\mathrm{Sp}} \cong 
\mathfrak{h}_{g,1} (12)_{\mathrm{Sp}}$, which is known to be isomorphic to $\mathbb{Q}^{650}$ 
for $g \ge 5$ (see \cite[Table 3]{mss5}, where 
$\dim \left(\mathfrak{h}_{3,1} (12)^{\mathrm{Sp}}\right)$ is wrongly written, 
the correct number is $354$). 
Every $\Symp{Q}$-invariant linear map 
$\mathfrak{h}_{g,1} (12) \to \mathbb{Q}$ factors through 
$\mathfrak{h}_{g,1} (12)_{\mathrm{Sp}}$ and 
the natural projection $\mathfrak{h}_{g,1} (12) \to \mathfrak{h}_{g,1} (12)_{\mathrm{Sp}}$ 
is regarded as  the projection onto $\mathfrak{h}_{g,1} (12)^{\mathrm{Sp}}$. 
Therefore we get a coordinate system of 
$\mathfrak{h}_{g,1} (12)^{\mathrm{Sp}}$ by finding 650 linearly independent 
$\Symp{Q}$-invariant linear maps $\mathfrak{h}_{g,1} (12) \to \mathbb{Q}$. 

Since $\mathfrak{h}_{g,1} (12)$ is an $\Symp{Q}$-submodule of 
$H \otimes \mathcal{L}_{13} (H) \subset H^{\otimes 14}$, we have 
$\mathfrak{h}_{g,1} (12)^{\mathrm{Sp}} \subset \big(H^{\otimes 14} \big)^{\mathrm{Sp}}$. 
A coordinate for $\big(H^{\otimes 14} \big)^{\mathrm{Sp}}$ is classically known and it is given as follows. 
Divide the set $\{a, b, c, \ldots, m, n\}$ of $14$ letters into $7$ pairs, say 
$(i_1 j_1) (i_2 j_2) \cdots (i_7 j_7)$. 
Then we consider the map 
\[\mu_{(i_1 j_1) (i_2 j_2) \cdots (i_7 j_7)}: H^{\otimes 14} \to \mathbb{Q}\]
defined by 
\[x_a \otimes x_b \otimes \cdots \otimes x_n \mapsto \mu (x_{i_1}, x_{j_1}) 
\mu (x_{i_2}, x_{j_2}) \cdots \mu(x_{i_7}, x_{j_7}).\]
Here we call this map a {\it multiple contraction}. It is $\Symp{Q}$-invariant since $\mu$ is so. 
We use multiple contractions restricted to $\mathfrak{h}_{g,1} (12)$ as coordinates of 
$\mathfrak{h}_{g,1} (12)^{\mathrm{Sp}}$. 
Note that they are invariant under the stabilization map $\hg \hookrightarrow \mathfrak{h}_{g+1,1}$. 

On the other hand, we use Lie spiders to express elements of $\mathfrak{h}_{g,1}$. 
A {\it Lie spider with $(k+2)$ legs} is defined by 
\begin{align*}
&S(u_1,u_2,u_3,\cdots, u_{k+2}) \\
&:= u_1\otimes [u_2,[u_3,[\cdots [u_{k+1},u_{k+2}]\cdots ]]]
+u_2\otimes [[u_3,[u_4,[\cdots [u_{k+1},u_{k+2}]\cdots ]]],u_1]\\
& \quad + u_3\otimes [ [u_4,[u_5,[\cdots [u_{k+1},u_{k+2}]\cdots ]]],[u_1,u_2]]+\cdots +u_{k+2}\otimes 
[[[\cdots [u_1,u_2],\cdots],u_{k}],u_{k+1}],
\end{align*}
where $u_i \in H$. It is known  (see \cite{levine}, for instance) that Lie spiders with $(k+2)$ legs belong to 
$\mathfrak{h}_{g,1}(k)$ and generate it. 

After calculating the pairings of a large amount of multiple contractions and Lie spiders, 
we see that the 650 multiple contractions $C_1, C_2, \ldots, C_{650}$ in Appendix \ref{append:detector}
are linearly independent. (Here we omit to display the corresponding 650 Lie spiders since we may 
use the data in Appendix \ref{append:bracket} together with the Lie spider in Subsection \ref{subsec:cocycleC}.)


\subsection{Computation of the bracket map}\label{subsec:bracketimage}
Since the bracket map 
\[[ \,\cdot\, ,\,\cdot \,]: 
\displaystyle\bigoplus_{i=1}^{11} \big( \mathfrak{h}_{g,1} (i) \otimes \mathfrak{h}_{g,1} (12-i)\big) \longrightarrow 
\mathfrak{h}_{g,1} (12)\]
is $\mathrm{Sp} (2g, \mathbb{Q})$-equivariant, it induces a linear map
\[[ \,\cdot\, ,\,\cdot \,]: 
\left(\displaystyle\bigoplus_{i=1}^{11} \big( \mathfrak{h}_{g,1} (i) \otimes 
\mathfrak{h}_{g,1} (12-i)\big)\right)^{\mathrm{Sp}} \longrightarrow \mathfrak{h}_{g,1} (12)^{\mathrm{Sp}}.\]
Recall that each of multiple contractions $\mu_{(i_1 j_1) (i_2 j_2) \cdots (i_7 j_7)}: \hg (12) \to \mathbb{Q}$ 
factors through $\mathfrak{h}_{g,1} (12)_{\mathrm{Sp}} 
\cong \mathfrak{h}_{g,1} (12)^{\mathrm{Sp}}$. Therefore as long as 
we use the coordinate system constructed above, we may work in 
the whole $\mathfrak{h}_{g,1} (12)$ without considering the projection 
onto the $\mathrm{Sp} (2g,\mathbb{Q})$-invariant part.

Although the general formula for the bracket map is a little bit complicated, 
there is a clear formula for the bracket of two Lie spiders, see \cite[Section 2.4.1]{cv0} for instance. 
After implementing a computer code for the formula, 
we compute the coordinates of the images of many brackets in 
$\mathfrak{h}_{g,1} (12)_{\mathrm{Sp}}$. 
As a result, we see that the 649 elements of 
$\displaystyle\sum_{i=1}^{11} [ \mathfrak{h}_{g,1} (i),  \mathfrak{h}_{g,1} (12-i)]$ 
exhibited in Appendix \ref{append:bracket} 
generate a $649$-dimensional subspace $W$ in $\mathfrak{h}_{g,1} (12)_{\mathrm{Sp}} \cong \mathbb{Q}^{650}$. 
The result is valid for $g \ge 8$ since we only use $a_1, b_1, \ldots, a_8, b_8$ 
in this computation.

\subsection{The map $C$}\label{subsec:mapC}

Next we look for a non-trivial $\mathrm{Sp} (2g, \mathbb{Q})$-invariant linear map 
$C: \mathfrak{h}_{g,1} (12) \to \mathbb{Q}$ annihilating $W$. 
For that we compute a linear relation over $W \cong \mathbb{Q}^{649}$ 
among our 650 multiple contractions. 
The result is given by a linear combination of 647 multiple contractions 
as in Appendix \ref{append:C}.

\subsection{Checking that $C$ is a non-trivial cocycle}\label{subsec:cocycleC}

The final step needs the heaviest computation. To show that $C$ is a cocycle, namely it is 
trivial on the image of the bracket map, 
we compute general formulas of bracket maps from each of 
$[\hg(1), \hg(11)]$, $[\hg(2), \hg(10)]$, $\ldots$, $[\hg(6), \hg(6)]$ to $\hg(12)_{\mathrm{Sp}}$   
in terms of $\mu$. The results are $0$ for all cases. This shows that $C$ is a cocycle for all $g \ge 2$. 
The non-triviality of $C$ follows from the explicit computation that 
\[C(S(a_1, b_1, a_1, a_1, a_1, a_1, a_2, a_1, b_1, b_1, b_1, b_1, b_1, b_2))=5832,\]
which is valid for all $g \ge 2$. 

\subsection{Implementation}\label{subsec:Mathematica}

The authors performed the above computations by using Mathematica. For that 
they wrote Mathematica codes whose core part computes  
the pairings of an element in $H^{\otimes 14}$ with multiple contractions. 
This is not so a difficult task if we use 
general {\it transformation $($replacement$)$ rules to symbolic expressions} in Mathematica. 
For example, an element $x_1 \otimes x_2 \otimes \cdots \otimes x_{14} \in H^{\otimes 14}$ with 
$x_i \in \{a_1, b_1, \ldots, a_g, b_g\}$ can be represented by 
the value {\tt ts[x1, x2, $\ldots$, x14]} 
of a function {\tt ts} with {\it no} definition assigned. 
We can treat general elements in $H^{\otimes 14}$ 
by just taking their linear combinations. 
Then the pairing of $x_1 \otimes x_2 \otimes \cdots \otimes x_{14}$ with 
the multiple contraction $\mu_{(ci)(dk)(el)(f n)(gm)(hj)}$, for example,  is given by 
the replacement rule 
\begin{center}
{\tt ts[a\_, b\_, c\_, d\_, e\_, f\_, g\_, h\_, i\_, j\_, k\_, l\_, m\_, n\_] :> \\
mu[c, i] mu[d,k] mu[e, l] mu[f, n] mu[g, m] mu[h, j]}
\end{center}
where {\tt mu} is the function which computes the value of the bilinear form $\mu$.

\section{Final remark}\label{sec:rem}

While an $\mathrm{Sp} (2g,\mathbb{Q})$-invariant map $\hg (12) \to \mathbb{Q}$ 
inducing the isomorphism $H_1 (\hg^+)_{12}^\mathrm{Sp} \cong \mathbb{Q}$ is unique up to scalar for $g \ge 8$, 
a description of this map using multiple contractions such as our cocycle $C$ is not so. 
It might be possible to obtain more simple expression than $C$ 
by taking another set of multiple contractions 
as a coordinate system of $\hg (12)^\mathrm{Sp} \cong \mathbb{Q}^{650}$. 

In the study of the structure of the Lie algebra $\hg^+$, the determination of 
the Lie subalgebra 
\[J=\bigoplus_{k = 1}^\infty J (k), \qquad J (k) \subset \hg (k)\] 
of $\hg^+$ generated by 
the degree $1$ part $\hg (1)$ has been considered to be important. 
It was shown by Hain \cite{hain} that this problem is the same as the determination 
of the rational image of the Johnson homomorphisms for subgroups of the mapping class groups of a surface. 
To this problem, Enomoto and Satoh \cite{es} provided the following powerful tool. 
They showed that the $\mathrm{Sp} (2g,\mathbb{Q})$-equivariant map obtained as the composition 
\[ES_k: \hg (k) \hookrightarrow H \otimes \mathcal{L}_{k+1} (H) \hookrightarrow H^{\otimes (k+2)} 
\xrightarrow{\mu \otimes (\mathrm{id}^{\otimes k})} H^{\otimes k} \to 
\big(H^{\otimes k}\big)_{\mathbb{Z}/k \mathbb{Z}}\]
is not trivial in general, but its restriction to $J (k)$ is trivial for any $k \ge 2$. Here 
$\big(H^{\otimes k}\big)_{\mathbb{Z}/k \mathbb{Z}}$ denotes the 
coinvariant quotient of $H^{\otimes k}$ with respect to 
the action of $\mathbb{Z}/k \mathbb{Z}$ as rotations of the entries. 
Consequently, the image of $ES_k$  gives a lower-bound estimate
of the gap between $J_k$ and $\hg (k)$. 
There are also important papers of Conant \cite{conant} and 
Conant-Kassabov \cite{ck} on this subject.

Using some data and Mathematica codes for the computation of the previous sections, 
the authors observed the following relationship between $ES_{12}$ and our cocycle $C$. 
\begin{thm}
For $g \ge 6$, the cocycle $C: \hg (12) \to \mathbb{Q}$ factors through $\Im ES_{12}$. 
\end{thm}
\begin{proof}[Sketch of Proof]
First we see that the image of $ES_{12}: \hg (12)^\mathrm{Sp} \to 
\big(H^{\otimes 12}\big)_{\mathbb{Z}/12 \mathbb{Z}}^\mathrm{Sp} \cong \mathbb{Q}^{897}$ is 
$284$-dimensional and take a basis of $\Ker ES_{12}$. Then we observe that the map $C$ is trivial 
on the $\Ker ES_{12}$ by applying $C$ to the basis. 
\end{proof}

By using this theorem, it is possible to give another description of the 
$\mathrm{Sp} (2g,\mathbb{Q})$-invariant map $C: \hg (12) \to \mathbb{Q}$
in the form $C= C' \circ ES_{12}$ with an $\mathrm{Sp} (2g,\mathbb{Q})$-invariant map 
$C': \mathrm{Im} ES_{12} \subset \big(H^{\otimes 12}\big)_{\mathbb{Z}/12 \mathbb{Z}} \to \mathbb{Q}$. 
Such a map $C'$ is described by chord diagrams with $6$ chords serving as a coordinate system of 
$\big(H^{\otimes 12}\big)_{\mathbb{Z}/12 \mathbb{Z}}^\mathrm{Sp} \cong \mathbb{Q}^{897}$. 
The details of the above computation will appear elsewhere.

\newpage
\appendix
\section{A coordinate system of $\mathfrak{h}_{\infty,1} (12)^{\mathrm{Sp}} \cong \mathbb{Q}^{650}$}\label{append:detector}

The authors found the following 650 multiple contractions 
\[(C_1, \ldots, C_{650}): \mathfrak{h}_{g,1} (12)^{\mathrm{Sp}} \longrightarrow \mathbb{Q}^{650}\]
forming a coordinate system of $\mathfrak{h}_{g,1} (12)^{\mathrm{Sp}}$ 
for $g \ge 5$ (see \cite[Theorem 1.2]{mss5}). 
The multiple contraction $C_m$ for $1 \le m \le 650$ is given below in the form
\[m \ \  (i_2 j_2)(i_3 j_3)(i_4 j_4)(i_5j_5)(i_6j_6)(i_7j_7),\]
which means the restriction of the $\mathrm{Sp}(2g,\mathbb{Q})$-invariant map 
\[\mu_{(ab) (i_2 j_2) (i_3 j_3) \cdots (i_7 j_7)}: H^{\otimes 14} \longrightarrow \mathbb{Q}\]
to the subspace $\mathfrak{h}_{g,1} (12)$. 

\bigskip
\noindent
{\tiny
\input{dataA.txt}

}

\newpage
\section{The image of the bracket map}\label{append:bracket}

The brackets of the following $649$ elements generate a $649$-dimensional subspace $W$ of 
$\mathfrak{h}_{g,1} (12)^{\mathrm{Sp}} \cong \mathbb{Q}^{650}$ for $g \ge 8$. 
In the computation, three types of bracket maps are used, where Type 1 and Type 2 concern with 
$[\hg (1), \hg (11)]$ and Type 3 with $[\hg(3), \hg(9)]$.

\bigskip
\noindent
\underline{Type 1} \quad 
$1$--$469: [S(a_1, b_1, a_8), S(b_8, X)]=S(a_1, b_1, X)$, where $X$ is a sequence  
obtained by permuting $\{a_2, b_2, \ldots, a_7, b_7\}$ as follows: 

{\tiny
\noindent
\input{dataB_1.txt}

}

\newpage
\noindent
\underline{Type 2} \quad 
$470$--$648: [S(a_1, a_2, a_3), S(b_3, Y, b_1)]$, where $Y$ is a word 
obtained by permuting $\{b_2, a_4, b_4, a_5, b_5,\ldots, a_8, b_8\}$ as follows: 

{\tiny
\noindent
\input{dataB_2.txt}

}

\bigskip
\noindent
\underline{Type 3} \quad 
$649: [S(b_1, a_4, a_3, a_1, a_2), S(b_2, a_5, a_6, b_5, a_7, b_6, a_8, b_3, b_7, b_8, b_4)]$

\newpage
\section{The cocycle $C$}\label{append:C}


The cocycle $C$ is given as a linear combination of 
$C_1, C_2, \ldots, C_{647}$ in Appendix \ref{append:detector}. 
The explicit formula is as follows. 

\bigskip
{\footnotesize
\input{dataC.txt}

}

\bibliographystyle{amsplain}

\end{document}

%% file: dataA.txt
\[\begin{array}{rrrr}
1 \ (ci) (dk) (el) (fn) (gm) (hj) &
2 \ (ci) (dk) (em) (fl) (gn) (hj) &
3 \ (ci) (dk) (el) (fm) (gn) (hj) &
4 \ (ci) (dj) (en) (fm) (gl) (hk) \\
5 \ (ci) (dj) (em) (fn) (gl) (hk) &
6 \ (ci) (dj) (en) (fl) (gm) (hk) &
7 \ (ci) (dj) (el) (fn) (gm) (hk) &
8 \ (ci) (dj) (em) (fl) (gn) (hk) \\
9 \ (ci) (dj) (el) (fm) (gn) (hk) &
10 \ (ci) (dj) (en) (fm) (gk) (hl) &
11 \ (ci) (dj) (em) (fn) (gk) (hl) &
12 \ (ci) (dj) (en) (fk) (gm) (hl) \\
13 \ (ci) (dk) (ej) (fn) (gm) (hl) &
14 \ (ci) (dj) (ek) (fn) (gm) (hl) &
15 \ (ci) (dk) (em) (fj) (gn) (hl) &
16 \ (ci) (dj) (em) (fk) (gn) (hl) \\
17 \ (ci) (dk) (ej) (fm) (gn) (hl) &
18 \ (ci) (dj) (ek) (fm) (gn) (hl) &
19 \ (ci) (dk) (el) (fn) (gj) (hm) &
20 \ (ci) (dj) (en) (fl) (gk) (hm) \\
21 \ (ci) (dj) (el) (fn) (gk) (hm) &
22 \ (ci) (dj) (en) (fk) (gl) (hm) &
23 \ (ci) (dk) (ej) (fn) (gl) (hm) &
24 \ (ci) (dj) (ek) (fn) (gl) (hm) \\
25 \ (ci) (dk) (el) (fj) (gn) (hm) &
26 \ (ci) (dj) (el) (fk) (gn) (hm) &
27 \ (ci) (dk) (ej) (fl) (gn) (hm) &
28 \ (ci) (dj) (ek) (fl) (gn) (hm) \\
29 \ (ci) (dk) (em) (fl) (gj) (hn) &
30 \ (ci) (dj) (em) (fl) (gk) (hn) &
31 \ (ci) (dj) (el) (fm) (gk) (hn) &
32 \ (ci) (dk) (em) (fj) (gl) (hn) \\
33 \ (ci) (dj) (em) (fk) (gl) (hn) &
34 \ (ci) (dj) (ek) (fm) (gl) (hn) &
35 \ (cd) (el) (fn) (gm) (hk) (ij) &
36 \ (cd) (en) (fk) (gm) (hl) (ij) \\
37 \ (cd) (ek) (fn) (gm) (hl) (ij) &
38 \ (cd) (ek) (fm) (gn) (hl) (ij) &
39 \ (cd) (el) (fn) (gk) (hm) (ij) &
40 \ (cd) (ek) (fn) (gl) (hm) (ij) \\
41 \ (cd) (ek) (fl) (gn) (hm) (ij) &
42 \ (cd) (el) (fn) (gm) (hj) (ik) &
43 \ (cd) (en) (fm) (gj) (hl) (ik) &
44 \ (cd) (em) (fn) (gj) (hl) (ik) \\
45 \ (cd) (en) (fj) (gm) (hl) (ik) &
46 \ (ce) (dj) (fn) (gm) (hl) (ik) &
47 \ (cd) (ej) (fn) (gm) (hl) (ik) &
48 \ (ce) (dj) (fm) (gn) (hl) (ik) \\
49 \ (cd) (ej) (fm) (gn) (hl) (ik) &
50 \ (cd) (el) (fn) (gj) (hm) (ik) &
51 \ (cd) (en) (fj) (gl) (hm) (ik) &
52 \ (ce) (dj) (fn) (gl) (hm) (ik) \\
53 \ (cd) (ej) (fn) (gl) (hm) (ik) &
54 \ (cd) (el) (fj) (gn) (hm) (ik) &
55 \ (ce) (dj) (fl) (gn) (hm) (ik) &
56 \ (cd) (ej) (fl) (gn) (hm) (ik) \\
57 \ (ce) (dj) (fm) (gl) (hn) (ik) &
58 \ (cd) (ej) (fm) (gl) (hn) (ik) &
59 \ (ce) (dj) (fl) (gm) (hn) (ik) &
60 \ (cd) (ej) (fl) (gm) (hn) (ik) \\
61 \ (cd) (en) (fk) (gm) (hj) (il) &
62 \ (cd) (ek) (fn) (gm) (hj) (il) &
63 \ (cd) (ek) (fm) (gn) (hj) (il) &
64 \ (cd) (en) (fj) (gm) (hk) (il) \\
65 \ (ce) (dj) (fn) (gm) (hk) (il) &
66 \ (cd) (ej) (fn) (gm) (hk) (il) &
67 \ (ce) (dj) (fm) (gn) (hk) (il) &
68 \ (cd) (ej) (fm) (gn) (hk) (il) \\
69 \ (cd) (ek) (fn) (gj) (hm) (il) &
70 \ (cd) (en) (fj) (gk) (hm) (il) &
71 \ (ce) (dj) (fn) (gk) (hm) (il) &
72 \ (cd) (ej) (fn) (gk) (hm) (il) \\
73 \ (cd) (ek) (fj) (gn) (hm) (il) &
74 \ (ce) (dj) (fk) (gn) (hm) (il) &
75 \ (cd) (ej) (fk) (gn) (hm) (il) &
76 \ (ce) (dj) (fm) (gk) (hn) (il) \\
77 \ (cd) (ej) (fm) (gk) (hn) (il) &
78 \ (cd) (ek) (fj) (gm) (hn) (il) &
79 \ (ce) (dj) (fk) (gm) (hn) (il) &
80 \ (cd) (ej) (fk) (gm) (hn) (il) \\
81 \ (cd) (el) (fn) (gk) (hj) (im) &
82 \ (cd) (ek) (fn) (gl) (hj) (im) &
83 \ (cd) (ek) (fl) (gn) (hj) (im) &
84 \ (cd) (el) (fn) (gj) (hk) (im) \\
85 \ (ce) (dj) (fn) (gl) (hk) (im) &
86 \ (cd) (ej) (fn) (gl) (hk) (im) &
87 \ (ce) (dj) (fl) (gn) (hk) (im) &
88 \ (cd) (ej) (fl) (gn) (hk) (im) \\
89 \ (cd) (ek) (fn) (gj) (hl) (im) &
90 \ (ce) (dj) (fn) (gk) (hl) (im) &
91 \ (cd) (ej) (fn) (gk) (hl) (im) &
92 \ (cd) (ek) (fj) (gn) (hl) (im) \\
93 \ (ce) (dj) (fk) (gn) (hl) (im) &
94 \ (cd) (ej) (fk) (gn) (hl) (im) &
95 \ (ce) (dj) (fl) (gk) (hn) (im) &
96 \ (cd) (ej) (fl) (gk) (hn) (im) \\
97 \ (ce) (dj) (fk) (gl) (hn) (im) &
98 \ (cd) (ej) (fk) (gl) (hn) (im) &
99 \ (ce) (dj) (fm) (gl) (hk) (in) &
100 \ (ce) (dj) (fl) (gm) (hk) (in) \\
101 \ (ce) (dj) (fm) (gk) (hl) (in) &
102 \ (ce) (dj) (fk) (gm) (hl) (in) &
103 \ (cd) (ej) (fk) (gm) (hl) (in) &
104 \ (ce) (dj) (fl) (gk) (hm) (in) \\
105 \ (ce) (dj) (fk) (gl) (hm) (in) &
106 \ (ce) (di) (fn) (gm) (hl) (jk) &
107 \ (cd) (ei) (fn) (gm) (hl) (jk) &
108 \ (ce) (di) (fm) (gn) (hl) (jk) \\
109 \ (cd) (ei) (fm) (gn) (hl) (jk) &
110 \ (ce) (di) (fn) (gl) (hm) (jk) &
111 \ (cd) (ei) (fn) (gl) (hm) (jk) &
112 \ (ce) (di) (fl) (gn) (hm) (jk) \\
113 \ (cd) (ei) (fl) (gn) (hm) (jk) &
114 \ (ce) (di) (fm) (gl) (hn) (jk) &
115 \ (cd) (ei) (fm) (gl) (hn) (jk) &
116 \ (ce) (di) (fl) (gm) (hn) (jk) \\
117 \ (cd) (ei) (fl) (gm) (hn) (jk) &
118 \ (cd) (en) (fm) (gh) (il) (jk) &
119 \ (cd) (em) (fn) (gh) (il) (jk) &
120 \ (cf) (dh) (en) (gm) (il) (jk) \\
121 \ (cd) (en) (fh) (gm) (il) (jk) &
122 \ (ce) (dh) (fn) (gm) (il) (jk) &
123 \ (cd) (eh) (fn) (gm) (il) (jk) &
124 \ (cf) (dh) (em) (gn) (il) (jk) \\
125 \ (ce) (dh) (fm) (gn) (il) (jk) &
126 \ (cd) (eh) (fm) (gn) (il) (jk) &
127 \ (cf) (dg) (en) (hm) (il) (jk) &
128 \ (cd) (en) (fg) (hm) (il) (jk) \\
129 \ (ce) (dg) (fn) (hm) (il) (jk) &
130 \ (cd) (eg) (fn) (hm) (il) (jk) &
131 \ (ce) (df) (gn) (hm) (il) (jk) &
132 \ (cf) (dg) (em) (hn) (il) (jk) \\
133 \ (ce) (dg) (fm) (hn) (il) (jk) &
134 \ (cd) (eg) (fm) (hn) (il) (jk) &
135 \ (ce) (df) (gm) (hn) (il) (jk) &
136 \ (cd) (el) (fn) (gh) (im) (jk) \\
137 \ (cf) (dh) (en) (gl) (im) (jk) &
138 \ (cd) (en) (fh) (gl) (im) (jk) &
139 \ (ce) (dh) (fn) (gl) (im) (jk) &
140 \ (cd) (eh) (fn) (gl) (im) (jk) \\
141 \ (cf) (dh) (el) (gn) (im) (jk) &
142 \ (cd) (el) (fh) (gn) (im) (jk) &
143 \ (ce) (dh) (fl) (gn) (im) (jk) &
144 \ (cd) (eh) (fl) (gn) (im) (jk) \\
145 \ (cf) (dg) (en) (hl) (im) (jk) &
146 \ (cd) (en) (fg) (hl) (im) (jk) &
147 \ (ce) (dg) (fn) (hl) (im) (jk) &
148 \ (cd) (eg) (fn) (hl) (im) (jk) \\
149 \ (ce) (df) (gn) (hl) (im) (jk) &
150 \ (cf) (dg) (el) (hn) (im) (jk) &
151 \ (ce) (dg) (fl) (hn) (im) (jk) &
152 \ (cd) (eg) (fl) (hn) (im) (jk) \\
153 \ (ce) (df) (gl) (hn) (im) (jk) &
154 \ (cd) (ef) (gl) (hn) (im) (jk) &
155 \ (cf) (dh) (em) (gl) (in) (jk) &
156 \ (ce) (dh) (fm) (gl) (in) (jk) \\
157 \ (cd) (eh) (fm) (gl) (in) (jk) &
158 \ (cf) (dh) (el) (gm) (in) (jk) &
159 \ (ce) (dh) (fl) (gm) (in) (jk) &
160 \ (cd) (eh) (fl) (gm) (in) (jk) \\
161 \ (cf) (dg) (em) (hl) (in) (jk) &
162 \ (ce) (dg) (fm) (hl) (in) (jk) &
163 \ (cd) (eg) (fm) (hl) (in) (jk) &
164 \ (ce) (df) (gm) (hl) (in) (jk) \\
165 \ (cf) (dl) (eg) (hm) (in) (jk) &
166 \ (cf) (dg) (el) (hm) (in) (jk) &
167 \ (ce) (dg) (fl) (hm) (in) (jk) &
168 \ (cd) (eg) (fl) (hm) (in) (jk) \\
169 \ (ce) (df) (gl) (hm) (in) (jk) &
170 \ (cd) (en) (fi) (gm) (hk) (jl) &
171 \ (ce) (di) (fn) (gm) (hk) (jl) &
172 \ (cd) (ei) (fn) (gm) (hk) (jl) 
\end{array}\]

\newpage
\[\begin{array}{rrrr}
173 \ (ce) (di) (fm) (gn) (hk) (jl) &
174 \ (cd) (ei) (fm) (gn) (hk) (jl) &
175 \ (cd) (en) (fi) (gk) (hm) (jl) &
176 \ (ce) (di) (fn) (gk) (hm) (jl) \\
177 \ (cd) (ei) (fn) (gk) (hm) (jl) &
178 \ (ce) (di) (fk) (gn) (hm) (jl) &
179 \ (cd) (ei) (fk) (gn) (hm) (jl) &
180 \ (ce) (di) (fm) (gk) (hn) (jl) \\
181 \ (cd) (ei) (fm) (gk) (hn) (jl) &
182 \ (cd) (ek) (fi) (gm) (hn) (jl) &
183 \ (ce) (di) (fk) (gm) (hn) (jl) &
184 \ (cd) (ei) (fk) (gm) (hn) (jl) \\
185 \ (cg) (dh) (en) (fm) (ik) (jl) &
186 \ (cg) (dh) (em) (fn) (ik) (jl) &
187 \ (cd) (en) (fm) (gh) (ik) (jl) &
188 \ (cd) (em) (fn) (gh) (ik) (jl) \\
189 \ (cf) (dh) (en) (gm) (ik) (jl) &
190 \ (cd) (en) (fh) (gm) (ik) (jl) &
191 \ (ce) (dh) (fn) (gm) (ik) (jl) &
192 \ (cd) (eh) (fn) (gm) (ik) (jl) \\
193 \ (cf) (dh) (em) (gn) (ik) (jl) &
194 \ (ce) (dh) (fm) (gn) (ik) (jl) &
195 \ (cd) (eh) (fm) (gn) (ik) (jl) &
196 \ (cf) (dg) (en) (hm) (ik) (jl) \\
197 \ (cd) (en) (fg) (hm) (ik) (jl) &
198 \ (ce) (dg) (fn) (hm) (ik) (jl) &
199 \ (cd) (eg) (fn) (hm) (ik) (jl) &
200 \ (ce) (df) (gn) (hm) (ik) (jl) \\
201 \ (cf) (dg) (em) (hn) (ik) (jl) &
202 \ (ce) (dg) (fm) (hn) (ik) (jl) &
203 \ (cd) (eg) (fm) (hn) (ik) (jl) &
204 \ (ce) (df) (gm) (hn) (ik) (jl) \\
205 \ (cd) (ek) (fn) (gh) (im) (jl) &
206 \ (cf) (dh) (en) (gk) (im) (jl) &
207 \ (cd) (en) (fh) (gk) (im) (jl) &
208 \ (ce) (dh) (fn) (gk) (im) (jl) \\
209 \ (cd) (eh) (fn) (gk) (im) (jl) &
210 \ (cf) (dh) (ek) (gn) (im) (jl) &
211 \ (cd) (ek) (fh) (gn) (im) (jl) &
212 \ (ce) (dh) (fk) (gn) (im) (jl) \\
213 \ (cd) (eh) (fk) (gn) (im) (jl) &
214 \ (cf) (dg) (en) (hk) (im) (jl) &
215 \ (cd) (en) (fg) (hk) (im) (jl) &
216 \ (ce) (dg) (fn) (hk) (im) (jl) \\
217 \ (cd) (eg) (fn) (hk) (im) (jl) &
218 \ (ce) (df) (gn) (hk) (im) (jl) &
219 \ (cd) (ef) (gn) (hk) (im) (jl) &
220 \ (cf) (dg) (ek) (hn) (im) (jl) \\
221 \ (cd) (ek) (fg) (hn) (im) (jl) &
222 \ (ce) (dg) (fk) (hn) (im) (jl) &
223 \ (cd) (eg) (fk) (hn) (im) (jl) &
224 \ (ce) (df) (gk) (hn) (im) (jl) \\
225 \ (cd) (ef) (gk) (hn) (im) (jl) &
226 \ (cf) (dh) (em) (gk) (in) (jl) &
227 \ (ce) (dh) (fm) (gk) (in) (jl) &
228 \ (cd) (eh) (fm) (gk) (in) (jl) \\
229 \ (cf) (dh) (ek) (gm) (in) (jl) &
230 \ (ce) (dh) (fk) (gm) (in) (jl) &
231 \ (cd) (eh) (fk) (gm) (in) (jl) &
232 \ (cf) (dm) (eg) (hk) (in) (jl) \\
233 \ (cf) (dg) (em) (hk) (in) (jl) &
234 \ (ce) (dg) (fm) (hk) (in) (jl) &
235 \ (cd) (eg) (fm) (hk) (in) (jl) &
236 \ (ce) (df) (gm) (hk) (in) (jl) \\
237 \ (cf) (dg) (ek) (hm) (in) (jl) &
238 \ (ce) (dg) (fk) (hm) (in) (jl) &
239 \ (cd) (eg) (fk) (hm) (in) (jl) &
240 \ (ce) (df) (gk) (hm) (in) (jl) \\
241 \ (cd) (ef) (gk) (hm) (in) (jl) &
242 \ (cd) (el) (fn) (gi) (hk) (jm) &
243 \ (ce) (di) (fn) (gl) (hk) (jm) &
244 \ (cd) (ei) (fn) (gl) (hk) (jm) \\
245 \ (ce) (di) (fl) (gn) (hk) (jm) &
246 \ (cd) (ei) (fl) (gn) (hk) (jm) &
247 \ (cd) (ek) (fn) (gi) (hl) (jm) &
248 \ (cf) (di) (en) (gk) (hl) (jm) \\
249 \ (ce) (di) (fn) (gk) (hl) (jm) &
250 \ (cd) (ei) (fn) (gk) (hl) (jm) &
251 \ (cf) (di) (ek) (gn) (hl) (jm) &
252 \ (ce) (di) (fk) (gn) (hl) (jm) \\
253 \ (cd) (ei) (fk) (gn) (hl) (jm) &
254 \ (cf) (di) (el) (gk) (hn) (jm) &
255 \ (ce) (di) (fl) (gk) (hn) (jm) &
256 \ (cd) (ei) (fl) (gk) (hn) (jm) \\
257 \ (cf) (di) (ek) (gl) (hn) (jm) &
258 \ (ce) (di) (fk) (gl) (hn) (jm) &
259 \ (cd) (ei) (fk) (gl) (hn) (jm) &
260 \ (cd) (el) (fn) (gh) (ik) (jm) \\
261 \ (cf) (dh) (en) (gl) (ik) (jm) &
262 \ (cd) (en) (fh) (gl) (ik) (jm) &
263 \ (ce) (dh) (fn) (gl) (ik) (jm) &
264 \ (cd) (eh) (fn) (gl) (ik) (jm) \\
265 \ (cf) (dh) (el) (gn) (ik) (jm) &
266 \ (cd) (el) (fh) (gn) (ik) (jm) &
267 \ (ce) (dh) (fl) (gn) (ik) (jm) &
268 \ (cd) (eh) (fl) (gn) (ik) (jm) \\
269 \ (cd) (en) (fg) (hl) (ik) (jm) &
270 \ (ce) (dg) (fn) (hl) (ik) (jm) &
271 \ (cd) (eg) (fn) (hl) (ik) (jm) &
272 \ (ce) (df) (gn) (hl) (ik) (jm) \\
273 \ (cf) (dg) (el) (hn) (ik) (jm) &
274 \ (ce) (dg) (fl) (hn) (ik) (jm) &
275 \ (cd) (eg) (fl) (hn) (ik) (jm) &
276 \ (ce) (df) (gl) (hn) (ik) (jm) \\
277 \ (cd) (ef) (gl) (hn) (ik) (jm) &
278 \ (cd) (ek) (fn) (gh) (il) (jm) &
279 \ (cf) (dh) (en) (gk) (il) (jm) &
280 \ (cd) (en) (fh) (gk) (il) (jm) \\
281 \ (ce) (dh) (fn) (gk) (il) (jm) &
282 \ (cd) (eh) (fn) (gk) (il) (jm) &
283 \ (cf) (dh) (ek) (gn) (il) (jm) &
284 \ (cd) (ek) (fh) (gn) (il) (jm) \\
285 \ (ce) (dh) (fk) (gn) (il) (jm) &
286 \ (cd) (eh) (fk) (gn) (il) (jm) &
287 \ (cf) (dg) (en) (hk) (il) (jm) &
288 \ (cd) (en) (fg) (hk) (il) (jm) \\
289 \ (ce) (dg) (fn) (hk) (il) (jm) &
290 \ (cd) (eg) (fn) (hk) (il) (jm) &
291 \ (ce) (df) (gn) (hk) (il) (jm) &
292 \ (cf) (dg) (ek) (hn) (il) (jm) \\
293 \ (ce) (dg) (fk) (hn) (il) (jm) &
294 \ (cd) (eg) (fk) (hn) (il) (jm) &
295 \ (ce) (df) (gk) (hn) (il) (jm) &
296 \ (cd) (ef) (gk) (hn) (il) (jm) \\
297 \ (cf) (dh) (el) (gk) (in) (jm) &
298 \ (ce) (dh) (fl) (gk) (in) (jm) &
299 \ (cd) (eh) (fl) (gk) (in) (jm) &
300 \ (cf) (dh) (ek) (gl) (in) (jm) \\
301 \ (ce) (dh) (fk) (gl) (in) (jm) &
302 \ (cd) (eh) (fk) (gl) (in) (jm) &
303 \ (cf) (dg) (el) (hk) (in) (jm) &
304 \ (ce) (dg) (fl) (hk) (in) (jm) \\
305 \ (cd) (eg) (fl) (hk) (in) (jm) &
306 \ (ce) (df) (gl) (hk) (in) (jm) &
307 \ (cf) (dk) (eg) (hl) (in) (jm) &
308 \ (cf) (dg) (ek) (hl) (in) (jm) \\
309 \ (ce) (dg) (fk) (hl) (in) (jm) &
310 \ (cd) (eg) (fk) (hl) (in) (jm) &
311 \ (ce) (df) (gk) (hl) (in) (jm) &
312 \ (cd) (ef) (gk) (hl) (in) (jm) \\
313 \ (ce) (di) (fm) (gl) (hk) (jn) &
314 \ (cf) (di) (el) (gm) (hk) (jn) &
315 \ (ce) (di) (fl) (gm) (hk) (jn) &
316 \ (ce) (di) (fm) (gk) (hl) (jn) \\
317 \ (ce) (di) (fk) (gm) (hl) (jn) &
318 \ (cd) (ei) (fk) (gm) (hl) (jn) &
319 \ (cf) (di) (el) (gk) (hm) (jn) &
320 \ (ce) (di) (fl) (gk) (hm) (jn) \\
321 \ (ce) (di) (fk) (gl) (hm) (jn) &
322 \ (cg) (dh) (em) (fl) (ik) (jn) &
323 \ (cg) (dh) (el) (fm) (ik) (jn) &
324 \ (cf) (dh) (em) (gl) (ik) (jn) \\
325 \ (ce) (dh) (fm) (gl) (ik) (jn) &
326 \ (cd) (eh) (fm) (gl) (ik) (jn) &
327 \ (cf) (dh) (el) (gm) (ik) (jn) &
328 \ (ce) (dh) (fl) (gm) (ik) (jn) \\
329 \ (cd) (eh) (fl) (gm) (ik) (jn) &
330 \ (cf) (dg) (em) (hl) (ik) (jn) &
331 \ (ce) (dg) (fm) (hl) (ik) (jn) &
332 \ (cd) (eg) (fm) (hl) (ik) (jn) \\
333 \ (ce) (df) (gm) (hl) (ik) (jn) &
334 \ (cf) (dl) (eg) (hm) (ik) (jn) &
335 \ (cf) (dg) (el) (hm) (ik) (jn) &
336 \ (ce) (dg) (fl) (hm) (ik) (jn) \\
337 \ (cd) (eg) (fl) (hm) (ik) (jn) &
338 \ (ce) (df) (gl) (hm) (ik) (jn) &
339 \ (cg) (dh) (ek) (fm) (il) (jn) &
340 \ (cf) (dh) (em) (gk) (il) (jn) \\
341 \ (ce) (dh) (fm) (gk) (il) (jn) &
342 \ (cd) (eh) (fm) (gk) (il) (jn) &
343 \ (cf) (dh) (ek) (gm) (il) (jn) &
344 \ (ce) (dh) (fk) (gm) (il) (jn) \\
345 \ (cd) (eh) (fk) (gm) (il) (jn) &
346 \ (cf) (dg) (em) (hk) (il) (jn) &
347 \ (ce) (dg) (fm) (hk) (il) (jn) &
348 \ (cd) (eg) (fm) (hk) (il) (jn) \\
349 \ (ce) (df) (gm) (hk) (il) (jn) &
350 \ (cf) (dk) (eg) (hm) (il) (jn) &
351 \ (cf) (dg) (ek) (hm) (il) (jn) &
352 \ (ce) (dg) (fk) (hm) (il) (jn) \\
353 \ (cd) (eg) (fk) (hm) (il) (jn) &
354 \ (ce) (df) (gk) (hm) (il) (jn) &
355 \ (cf) (dh) (el) (gk) (im) (jn) &
356 \ (ce) (dh) (fl) (gk) (im) (jn) \\
357 \ (cd) (eh) (fl) (gk) (im) (jn) &
358 \ (cf) (dk) (eh) (gl) (im) (jn) &
359 \ (cf) (dh) (ek) (gl) (im) (jn) &
360 \ (ce) (dh) (fk) (gl) (im) (jn) \\
361 \ (cd) (eh) (fk) (gl) (im) (jn) &
362 \ (cf) (dl) (eg) (hk) (im) (jn) &
363 \ (cf) (dg) (el) (hk) (im) (jn) &
364 \ (ce) (dg) (fl) (hk) (im) (jn) \\
365 \ (cd) (eg) (fl) (hk) (im) (jn) &
366 \ (ce) (df) (gl) (hk) (im) (jn) &
367 \ (cf) (dk) (eg) (hl) (im) (jn) &
368 \ (cf) (dg) (ek) (hl) (im) (jn) \\
369 \ (ce) (dg) (fk) (hl) (im) (jn) &
370 \ (cd) (eg) (fk) (hl) (im) (jn) &
371 \ (ce) (df) (gk) (hl) (im) (jn) &
372 \ (cd) (ej) (fn) (gm) (hi) (kl) \\
373 \ (cd) (ej) (fm) (gn) (hi) (kl) &
374 \ (cd) (en) (fi) (gm) (hj) (kl) &
375 \ (cd) (ei) (fn) (gm) (hj) (kl) &
376 \ (cd) (ei) (fm) (gn) (hj) (kl) \\
377 \ (cd) (ej) (fn) (gi) (hm) (kl) &
378 \ (cd) (en) (fi) (gj) (hm) (kl) &
379 \ (cd) (ei) (fn) (gj) (hm) (kl) &
380 \ (cf) (di) (ej) (gn) (hm) (kl) \\
381 \ (cd) (ej) (fi) (gn) (hm) (kl) &
382 \ (cd) (ei) (fj) (gn) (hm) (kl) &
383 \ (cd) (ej) (fm) (gi) (hn) (kl) &
384 \ (cd) (ei) (fm) (gj) (hn) (kl) \\
385 \ (cf) (di) (ej) (gm) (hn) (kl) &
386 \ (cd) (ej) (fi) (gm) (hn) (kl) &
387 \ (cd) (ei) (fj) (gm) (hn) (kl) &
388 \ (cd) (en) (fh) (gj) (im) (kl) \\
389 \ (cd) (eh) (fn) (gj) (im) (kl) &
390 \ (cf) (dh) (ej) (gn) (im) (kl) &
391 \ (ce) (dh) (fj) (gn) (im) (kl) &
392 \ (cd) (eh) (fj) (gn) (im) (kl) \\
393 \ (cd) (eg) (fn) (hj) (im) (kl) &
394 \ (cd) (ef) (gn) (hj) (im) (kl) &
395 \ (cf) (dg) (ej) (hn) (im) (kl) &
396 \ (cd) (eg) (fj) (hn) (im) (kl) \\
397 \ (cd) (ef) (gj) (hn) (im) (kl) &
398 \ (ce) (dh) (fm) (gj) (in) (kl) &
399 \ (cd) (eh) (fm) (gj) (in) (kl) &
400 \ (cf) (dh) (ej) (gm) (in) (kl) \\
401 \ (ce) (dh) (fj) (gm) (in) (kl) &
402 \ (cd) (eh) (fj) (gm) (in) (kl) &
403 \ (cd) (eg) (fm) (hj) (in) (kl) &
404 \ (cf) (dg) (ej) (hm) (in) (kl) \\
405 \ (cd) (ej) (fg) (hm) (in) (kl) &
406 \ (cd) (eg) (fj) (hm) (in) (kl) &
407 \ (cd) (ef) (gj) (hm) (in) (kl) &
408 \ (cd) (ei) (fn) (gh) (jm) (kl) \\
409 \ (cd) (en) (fh) (gi) (jm) (kl) &
410 \ (ce) (dh) (fn) (gi) (jm) (kl) &
411 \ (cd) (eh) (fn) (gi) (jm) (kl) &
412 \ (cf) (di) (eh) (gn) (jm) (kl) 
\end{array}\]

\newpage
\[\begin{array}{rrrr}
413 \ (cf) (dh) (ei) (gn) (jm) (kl) &
414 \ (ce) (di) (fh) (gn) (jm) (kl) &
415 \ (cd) (ei) (fh) (gn) (jm) (kl) &
416 \ (ce) (dh) (fi) (gn) (jm) (kl) \\
417 \ (cd) (eh) (fi) (gn) (jm) (kl) &
418 \ (ce) (dg) (fn) (hi) (jm) (kl) &
419 \ (cd) (eg) (fn) (hi) (jm) (kl) &
420 \ (ce) (df) (gn) (hi) (jm) (kl) \\
421 \ (cd) (ef) (gn) (hi) (jm) (kl) &
422 \ (cf) (di) (eg) (hn) (jm) (kl) &
423 \ (cf) (dg) (ei) (hn) (jm) (kl) &
424 \ (ce) (dg) (fi) (hn) (jm) (kl) \\
425 \ (cd) (eg) (fi) (hn) (jm) (kl) &
426 \ (ce) (df) (gi) (hn) (jm) (kl) &
427 \ (cd) (ef) (gi) (hn) (jm) (kl) &
428 \ (cf) (dh) (eg) (in) (jm) (kl) \\
429 \ (cf) (dg) (eh) (in) (jm) (kl) &
430 \ (cd) (eh) (fg) (in) (jm) (kl) &
431 \ (cd) (eg) (fh) (in) (jm) (kl) &
432 \ (ce) (dh) (fm) (gi) (jn) (kl) \\
433 \ (cd) (eh) (fm) (gi) (jn) (kl) &
434 \ (cf) (di) (eh) (gm) (jn) (kl) &
435 \ (cf) (dh) (ei) (gm) (jn) (kl) &
436 \ (ce) (di) (fh) (gm) (jn) (kl) \\
437 \ (cd) (ei) (fh) (gm) (jn) (kl) &
438 \ (ce) (dh) (fi) (gm) (jn) (kl) &
439 \ (cd) (eh) (fi) (gm) (jn) (kl) &
440 \ (ce) (dg) (fm) (hi) (jn) (kl) \\
441 \ (cd) (eg) (fm) (hi) (jn) (kl) &
442 \ (ce) (df) (gm) (hi) (jn) (kl) &
443 \ (cf) (di) (eg) (hm) (jn) (kl) &
444 \ (cf) (dg) (ei) (hm) (jn) (kl) \\
445 \ (cd) (ei) (fg) (hm) (jn) (kl) &
446 \ (ce) (dg) (fi) (hm) (jn) (kl) &
447 \ (cd) (eg) (fi) (hm) (jn) (kl) &
448 \ (ce) (df) (gi) (hm) (jn) (kl) \\
449 \ (cd) (ef) (gi) (hm) (jn) (kl) &
450 \ (cf) (dh) (eg) (im) (jn) (kl) &
451 \ (cf) (dg) (eh) (im) (jn) (kl) &
452 \ (cd) (eh) (fg) (im) (jn) (kl) \\
453 \ (cd) (eg) (fh) (im) (jn) (kl) &
454 \ (cd) (ej) (fn) (gl) (hi) (km) &
455 \ (cd) (ej) (fl) (gn) (hi) (km) &
456 \ (cd) (el) (fn) (gi) (hj) (km) \\
457 \ (cd) (ei) (fn) (gl) (hj) (km) &
458 \ (cd) (ei) (fl) (gn) (hj) (km) &
459 \ (cd) (ej) (fn) (gi) (hl) (km) &
460 \ (cd) (ei) (fn) (gj) (hl) (km) \\
461 \ (cf) (dj) (ei) (gn) (hl) (km) &
462 \ (cf) (di) (ej) (gn) (hl) (km) &
463 \ (cd) (ej) (fi) (gn) (hl) (km) &
464 \ (ce) (di) (fj) (gn) (hl) (km) \\
465 \ (cd) (ei) (fj) (gn) (hl) (km) &
466 \ (cd) (ej) (fl) (gi) (hn) (km) &
467 \ (cd) (ei) (fl) (gj) (hn) (km) &
468 \ (cf) (dj) (ei) (gl) (hn) (km) \\
469 \ (cf) (di) (ej) (gl) (hn) (km) &
470 \ (cd) (ej) (fi) (gl) (hn) (km) &
471 \ (ce) (di) (fj) (gl) (hn) (km) &
472 \ (cd) (ei) (fj) (gl) (hn) (km) \\
473 \ (cd) (eg) (fl) (hn) (ij) (km) &
474 \ (cd) (ef) (gl) (hn) (ij) (km) &
475 \ (cg) (dh) (ej) (fn) (il) (km) &
476 \ (cd) (eh) (fn) (gj) (il) (km) \\
477 \ (cf) (dj) (eh) (gn) (il) (km) &
478 \ (cf) (dh) (ej) (gn) (il) (km) &
479 \ (ce) (dh) (fj) (gn) (il) (km) &
480 \ (cd) (eh) (fj) (gn) (il) (km) \\
481 \ (ce) (dg) (fn) (hj) (il) (km) &
482 \ (cd) (eg) (fn) (hj) (il) (km) &
483 \ (ce) (df) (gn) (hj) (il) (km) &
484 \ (cf) (dj) (eg) (hn) (il) (km) \\
485 \ (cf) (dg) (ej) (hn) (il) (km) &
486 \ (ce) (dg) (fj) (hn) (il) (km) &
487 \ (cd) (eg) (fj) (hn) (il) (km) &
488 \ (ce) (df) (gj) (hn) (il) (km) \\
489 \ (cd) (ef) (gj) (hn) (il) (km) &
490 \ (ce) (dh) (fl) (gj) (in) (km) &
491 \ (cd) (eh) (fl) (gj) (in) (km) &
492 \ (cf) (dj) (eh) (gl) (in) (km) \\
493 \ (cf) (dh) (ej) (gl) (in) (km) &
494 \ (ce) (dh) (fj) (gl) (in) (km) &
495 \ (cd) (eh) (fj) (gl) (in) (km) &
496 \ (cd) (eg) (fl) (hj) (in) (km) \\
497 \ (cf) (dj) (eg) (hl) (in) (km) &
498 \ (cf) (dg) (ej) (hl) (in) (km) &
499 \ (cd) (ej) (fg) (hl) (in) (km) &
500 \ (ce) (dg) (fj) (hl) (in) (km) \\
501 \ (cd) (eg) (fj) (hl) (in) (km) &
502 \ (ce) (df) (gj) (hl) (in) (km) &
503 \ (cd) (ef) (gj) (hl) (in) (km) &
504 \ (cg) (dh) (ei) (fn) (jl) (km) \\
505 \ (cd) (ei) (fn) (gh) (jl) (km) &
506 \ (cd) (eh) (fn) (gi) (jl) (km) &
507 \ (cf) (di) (eh) (gn) (jl) (km) &
508 \ (cf) (dh) (ei) (gn) (jl) (km) \\
509 \ (ce) (di) (fh) (gn) (jl) (km) &
510 \ (cd) (ei) (fh) (gn) (jl) (km) &
511 \ (ce) (dh) (fi) (gn) (jl) (km) &
512 \ (cd) (eh) (fi) (gn) (jl) (km) \\
513 \ (cd) (ef) (gn) (hi) (jl) (km) &
514 \ (cf) (di) (eg) (hn) (jl) (km) &
515 \ (cf) (dg) (ei) (hn) (jl) (km) &
516 \ (cd) (eg) (fi) (hn) (jl) (km) \\
517 \ (cd) (ef) (gi) (hn) (jl) (km) &
518 \ (cf) (dh) (eg) (in) (jl) (km) &
519 \ (cf) (dg) (eh) (in) (jl) (km) &
520 \ (cd) (eh) (fg) (in) (jl) (km) \\
521 \ (cd) (eg) (fh) (in) (jl) (km) &
522 \ (ce) (dh) (fl) (gi) (jn) (km) &
523 \ (cd) (eh) (fl) (gi) (jn) (km) &
524 \ (cf) (di) (eh) (gl) (jn) (km) \\
525 \ (cf) (dh) (ei) (gl) (jn) (km) &
526 \ (ce) (di) (fh) (gl) (jn) (km) &
527 \ (ce) (dh) (fi) (gl) (jn) (km) &
528 \ (cd) (eh) (fi) (gl) (jn) (km) \\
529 \ (ce) (df) (gl) (hi) (jn) (km) &
530 \ (cf) (di) (eg) (hl) (jn) (km) &
531 \ (cf) (dg) (ei) (hl) (jn) (km) &
532 \ (cd) (eg) (fi) (hl) (jn) (km) \\
533 \ (ce) (df) (gi) (hl) (jn) (km) &
534 \ (cd) (ef) (gi) (hl) (jn) (km) &
535 \ (cf) (dh) (eg) (il) (jn) (km) &
536 \ (cf) (dg) (eh) (il) (jn) (km) \\
537 \ (cd) (eh) (fg) (il) (jn) (km) &
538 \ (ce) (dg) (fh) (il) (jn) (km) &
539 \ (cd) (eg) (fh) (il) (jn) (km) &
540 \ (ce) (df) (gh) (il) (jn) (km) \\
541 \ (cd) (ef) (gh) (il) (jn) (km) &
542 \ (cf) (di) (ej) (gm) (hl) (kn) &
543 \ (ce) (di) (fj) (gm) (hl) (kn) &
544 \ (ce) (di) (fj) (gl) (hm) (kn) \\
545 \ (cd) (eg) (fl) (hm) (ij) (kn) &
546 \ (cf) (dh) (ej) (gm) (il) (kn) &
547 \ (ce) (dh) (fj) (gm) (il) (kn) &
548 \ (cd) (eh) (fj) (gm) (il) (kn) \\
549 \ (ce) (dg) (fm) (hj) (il) (kn) &
550 \ (cd) (eg) (fm) (hj) (il) (kn) &
551 \ (ce) (df) (gm) (hj) (il) (kn) &
552 \ (cf) (dg) (ej) (hm) (il) (kn) \\
553 \ (ce) (dg) (fj) (hm) (il) (kn) &
554 \ (cd) (eg) (fj) (hm) (il) (kn) &
555 \ (ce) (df) (gj) (hm) (il) (kn) &
556 \ (ce) (dh) (fl) (gj) (im) (kn) \\
557 \ (ce) (dh) (fj) (gl) (im) (kn) &
558 \ (cd) (eg) (fl) (hj) (im) (kn) &
559 \ (cf) (dg) (ej) (hl) (im) (kn) &
560 \ (ce) (dg) (fj) (hl) (im) (kn) \\
561 \ (cd) (eg) (fj) (hl) (im) (kn) &
562 \ (ce) (df) (gj) (hl) (im) (kn) &
563 \ (ce) (dh) (fm) (gi) (jl) (kn) &
564 \ (cf) (di) (eh) (gm) (jl) (kn) \\
565 \ (cf) (dh) (ei) (gm) (jl) (kn) &
566 \ (ce) (di) (fh) (gm) (jl) (kn) &
567 \ (ce) (dh) (fi) (gm) (jl) (kn) &
568 \ (cd) (eh) (fi) (gm) (jl) (kn) \\
569 \ (cf) (dg) (ei) (hm) (jl) (kn) &
570 \ (cd) (eg) (fi) (hm) (jl) (kn) &
571 \ (ce) (df) (gi) (hm) (jl) (kn) &
572 \ (cd) (ef) (gi) (hm) (jl) (kn) \\
573 \ (cf) (dg) (eh) (im) (jl) (kn) &
574 \ (cd) (eh) (fg) (im) (jl) (kn) &
575 \ (cd) (eg) (fh) (im) (jl) (kn) &
576 \ (ce) (dh) (fl) (gi) (jm) (kn) \\
577 \ (ce) (di) (fh) (gl) (jm) (kn) &
578 \ (ce) (dh) (fi) (gl) (jm) (kn) &
579 \ (ce) (df) (gl) (hi) (jm) (kn) &
580 \ (ce) (dg) (fi) (hl) (jm) (kn) \\
581 \ (cd) (eg) (fi) (hl) (jm) (kn) &
582 \ (ce) (df) (gi) (hl) (jm) (kn) &
583 \ (cd) (ef) (gi) (hl) (jm) (kn) &
584 \ (cd) (eh) (fg) (il) (jm) (kn) \\
585 \ (cd) (eg) (fh) (il) (jm) (kn) &
586 \ (cd) (ej) (fn) (gk) (hi) (lm) &
587 \ (cd) (ej) (fk) (gn) (hi) (lm) &
588 \ (cd) (ei) (fn) (gk) (hj) (lm) \\
589 \ (cd) (ei) (fk) (gn) (hj) (lm) &
590 \ (cd) (ej) (fn) (gi) (hk) (lm) &
591 \ (cd) (ei) (fn) (gj) (hk) (lm) &
592 \ (cd) (ej) (fi) (gn) (hk) (lm) \\
593 \ (cd) (ei) (fj) (gn) (hk) (lm) &
594 \ (cd) (ej) (fk) (gi) (hn) (lm) &
595 \ (cd) (ei) (fk) (gj) (hn) (lm) &
596 \ (cd) (ej) (fi) (gk) (hn) (lm) \\
597 \ (cd) (ei) (fj) (gk) (hn) (lm) &
598 \ (cd) (eg) (fn) (hk) (ij) (lm) &
599 \ (cd) (eg) (fk) (hn) (ij) (lm) &
600 \ (cd) (ef) (gk) (hn) (ij) (lm) \\
601 \ (cd) (eh) (fn) (gj) (ik) (lm) &
602 \ (cd) (eh) (fj) (gn) (ik) (lm) &
603 \ (cd) (eg) (fn) (hj) (ik) (lm) &
604 \ (ce) (dg) (fj) (hn) (ik) (lm) \\
605 \ (cd) (eg) (fj) (hn) (ik) (lm) &
606 \ (ce) (df) (gj) (hn) (ik) (lm) &
607 \ (cd) (ef) (gj) (hn) (ik) (lm) &
608 \ (cd) (eh) (fk) (gj) (in) (lm) \\
609 \ (ce) (dh) (fj) (gk) (in) (lm) &
610 \ (cd) (eh) (fj) (gk) (in) (lm) &
611 \ (cd) (eg) (fk) (hj) (in) (lm) &
612 \ (cd) (ef) (gk) (hj) (in) (lm) \\
613 \ (ce) (dg) (fj) (hk) (in) (lm) &
614 \ (cd) (eg) (fj) (hk) (in) (lm) &
615 \ (ce) (df) (gj) (hk) (in) (lm) &
616 \ (cd) (ef) (gj) (hk) (in) (lm) \\
617 \ (cd) (eg) (fi) (hn) (jk) (lm) &
618 \ (cd) (eg) (fh) (in) (jk) (lm) &
619 \ (cd) (eh) (fk) (gi) (jn) (lm) &
620 \ (ce) (dh) (fi) (gk) (jn) (lm) \\
621 \ (cd) (eh) (fi) (gk) (jn) (lm) &
622 \ (cd) (eg) (fk) (hi) (jn) (lm) &
623 \ (ce) (dg) (fi) (hk) (jn) (lm) &
624 \ (cd) (eg) (fi) (hk) (jn) (lm) \\
625 \ (ce) (df) (gi) (hk) (jn) (lm) &
626 \ (cd) (ef) (gi) (hk) (jn) (lm) &
627 \ (cd) (eg) (fh) (ik) (jn) (lm) &
628 \ (cd) (ef) (gh) (ik) (jn) (lm) \\
629 \ (cd) (eg) (fj) (hi) (kn) (lm) &
630 \ (ce) (df) (gj) (hi) (kn) (lm) &
631 \ (ce) (dg) (fi) (hj) (kn) (lm) &
632 \ (cd) (eg) (fi) (hj) (kn) (lm) \\
633 \ (ce) (df) (gi) (hj) (kn) (lm) &
634 \ (cd) (ef) (gi) (hj) (kn) (lm) &
635 \ (cd) (eg) (fh) (ij) (kn) (lm) &
636 \ (ce) (di) (fm) (gj) (hk) (ln) \\
637 \ (ce) (di) (fj) (gm) (hk) (ln) &
638 \ (ce) (dh) (fk) (gm) (ij) (ln) &
639 \ (ce) (dg) (fj) (hm) (ik) (ln) &
640 \ (cd) (eg) (fj) (hm) (ik) (ln) \\
641 \ (ce) (df) (gj) (hm) (ik) (ln) &
642 \ (ce) (dg) (fj) (hk) (im) (ln) &
643 \ (ce) (df) (gj) (hk) (im) (ln) &
644 \ (cd) (eg) (fi) (hm) (jk) (ln) \\
645 \ (cd) (eg) (fh) (im) (jk) (ln) &
646 \ (ce) (df) (gi) (hk) (jm) (ln) &
647 \ (cd) (eg) (fh) (ik) (jm) (ln) &
648 \ (cf) (dg) (en) (hl) (ik) (jm) \\
649 \ (cg) (dh) (en) (fi) (jl) (km) &
650 \ (ci) (dk) (em) (fn) (gl) (hj) &
\end{array}\]

%% file: dataB_1.txt
\begin{align*}
&\text{$1$--$102$}\\
& \{a_2, a_3, b_2, b_3, a_4, b_4, a_5, a_6, b_6, b_5, a_7, b_7\}, \  
\{a_2, a_3, b_2, b_3, a_4, b_4, a_5, a_6, a_7, b_5, b_6, b_7\}, \  
\{a_2, a_3, b_2, b_3, a_4, a_5, b_4, a_6, b_5, b_6, a_7, b_7\}, \\  
& \{a_2, a_3, b_2, b_3, a_4, a_5, b_4, a_6, b_6, b_5, a_7, b_7\}, \  
\{a_2, a_3, b_2, b_3, a_4, a_5, b_4, a_6, a_7, b_5, b_6, b_7\}, \  
\{a_2, a_3, b_2, b_3, a_4, a_5, a_6, b_4, b_5, b_6, a_7, b_7\}, \\  
& \{a_2, a_3, b_2, b_3, a_4, a_5, a_6, b_4, b_6, b_5, a_7, b_7\}, \  
\{a_2, a_3, b_2, b_3, a_4, a_5, a_6, b_4, a_7, b_5, b_6, b_7\}, \  
\{a_2, a_3, b_2, b_3, a_4, a_5, a_6, b_5, b_4, b_6, a_7, b_7\}, \\  
& \{a_2, a_3, b_2, a_4, b_3, b_4, a_5, a_6, b_6, b_5, a_7, b_7\}, \  
\{a_2, a_3, b_2, a_4, b_3, b_4, a_5, a_6, a_7, b_5, b_6, b_7\}, \  
\{a_2, a_3, b_2, a_4, b_3, a_5, b_4, a_6, b_6, b_5, a_7, b_7\}, \\  
& \{a_2, a_3, b_2, a_4, b_3, a_5, b_4, a_6, a_7, b_5, b_6, b_7\}, \  
\{a_2, a_3, b_2, a_4, b_3, a_5, a_6, b_4, b_5, b_6, a_7, b_7\}, \  
\{a_2, a_3, b_2, a_4, b_3, a_5, a_6, b_4, b_6, b_5, a_7, b_7\}, \\  
& \{a_2, a_3, b_2, a_4, b_3, a_5, a_6, b_4, a_7, b_5, b_6, b_7\}, \  
\{a_2, a_3, b_2, a_4, b_3, a_5, b_5, a_6, b_4, b_6, a_7, b_7\}, \  
\{a_2, a_3, b_2, a_4, b_3, a_5, a_6, b_5, b_4, b_6, a_7, b_7\}, \\  
& \{a_2, a_3, b_2, a_4, b_3, a_5, a_6, a_7, b_4, b_5, b_6, b_7\}, \  
\{a_2, a_3, b_2, a_4, b_3, a_5, b_5, a_6, b_6, b_4, a_7, b_7\}, \  
\{a_2, a_3, b_2, a_4, b_3, a_5, b_5, a_6, a_7, b_4, b_6, b_7\}, \\  
& \{a_2, a_3, b_2, a_4, b_3, a_5, a_6, b_5, a_7, b_4, b_6, b_7\}, \  
\{a_2, a_3, b_2, a_4, b_4, b_3, a_5, a_6, b_6, b_5, a_7, b_7\}, \  
\{a_2, a_3, b_2, a_4, b_4, b_3, a_5, a_6, a_7, b_5, b_6, b_7\}, \\  
& \{a_2, a_3, b_2, a_4, a_5, b_3, b_4, a_6, b_6, b_5, a_7, b_7\}, \  
\{a_2, a_3, b_2, a_4, a_5, b_3, b_4, a_6, a_7, b_5, b_6, b_7\}, \  
\{a_2, a_3, b_2, a_4, a_5, b_3, a_6, b_4, b_5, b_6, a_7, b_7\}, \\  
& \{a_2, a_3, b_2, a_4, a_5, b_3, a_6, b_4, b_6, b_5, a_7, b_7\}, \  
\{a_2, a_3, b_2, a_4, a_5, b_3, a_6, b_4, a_7, b_5, b_6, b_7\}, \  
\{a_2, a_3, b_2, a_4, a_5, b_3, b_5, a_6, b_4, b_6, a_7, b_7\}, \\  
& \{a_2, a_3, b_2, a_4, a_5, b_3, a_6, b_5, b_4, b_6, a_7, b_7\}, \  
\{a_2, a_3, b_2, a_4, a_5, b_3, a_6, b_6, b_4, b_5, a_7, b_7\}, \  
\{a_2, a_3, b_2, a_4, a_5, b_3, a_6, a_7, b_4, b_5, b_6, b_7\}, \\  
& \{a_2, a_3, b_2, a_4, a_5, b_3, a_6, a_7, b_4, b_6, b_5, b_7\}, \  
\{a_2, a_3, b_2, a_4, a_5, b_3, b_5, a_6, b_6, b_4, a_7, b_7\}, \  
\{a_2, a_3, b_2, a_4, a_5, b_3, b_5, a_6, a_7, b_4, b_6, b_7\}, \\  
& \{a_2, a_3, b_2, a_4, a_5, b_3, a_6, b_5, b_6, b_4, a_7, b_7\}, \  
\{a_2, a_3, b_2, a_4, a_5, b_3, a_6, b_5, a_7, b_4, b_6, b_7\}, \  
\{a_2, a_3, b_2, a_4, b_4, a_5, b_3, a_6, b_6, b_5, a_7, b_7\}, \\  
& \{a_2, a_3, b_2, a_4, b_4, a_5, b_3, a_6, a_7, b_5, b_6, b_7\}, \  
\{a_2, a_3, b_2, a_4, a_5, b_4, b_3, a_6, b_6, b_5, a_7, b_7\}, \  
\{a_2, a_3, b_2, a_4, a_5, b_4, b_3, a_6, a_7, b_5, b_6, b_7\}, \\  
& \{a_2, a_3, b_2, a_4, a_5, a_6, b_3, b_4, b_5, b_6, a_7, b_7\}, \  
\{a_2, a_3, b_2, a_4, a_5, a_6, b_3, b_4, b_6, b_5, a_7, b_7\}, \  
\{a_2, a_3, b_2, a_4, a_5, a_6, b_3, b_4, a_7, b_5, b_6, b_7\}, \\  
& \{a_2, a_3, b_2, a_4, a_5, a_6, b_3, b_5, b_4, b_6, a_7, b_7\}, \  
\{a_2, a_3, b_2, a_4, a_5, a_6, b_3, b_6, b_4, b_5, a_7, b_7\}, \  
\{a_2, a_3, b_2, a_4, a_5, a_6, b_3, a_7, b_4, b_5, b_6, b_7\}, \\  
& \{a_2, a_3, b_2, a_4, a_5, a_6, b_3, a_7, b_4, b_6, b_5, b_7\}, \  
\{a_2, a_3, b_2, a_4, a_5, b_5, b_3, a_6, b_6, b_4, a_7, b_7\}, \  
\{a_2, a_3, b_2, a_4, a_5, b_5, b_3, a_6, a_7, b_4, b_6, b_7\}, \\  
& \{a_2, a_3, b_2, a_4, a_5, a_6, b_3, b_5, b_6, b_4, a_7, b_7\}, \  
\{a_2, a_3, b_2, a_4, a_5, a_6, b_3, b_5, a_7, b_4, b_6, b_7\}, \  
\{a_2, a_3, b_2, a_4, a_5, a_6, b_3, b_6, b_5, b_4, a_7, b_7\}, \\  
& \{a_2, a_3, b_2, a_4, a_5, a_6, b_3, a_7, b_5, b_4, b_6, b_7\}, \  
\{a_2, a_3, b_2, a_4, a_5, a_6, b_3, a_7, b_5, b_6, b_4, b_7\}, \  
\{a_2, a_3, b_2, a_4, b_4, a_5, a_6, b_3, b_5, b_6, a_7, b_7\}, \\  
& \{a_2, a_3, b_2, a_4, b_4, a_5, a_6, b_3, b_6, b_5, a_7, b_7\}, \  
\{a_2, a_3, b_2, a_4, b_4, a_5, a_6, b_3, a_7, b_5, b_6, b_7\}, \  
\{a_2, a_3, b_2, a_4, a_5, b_4, a_6, b_3, b_5, b_6, a_7, b_7\}, \\  
& \{a_2, a_3, b_2, a_4, a_5, b_4, a_6, b_3, b_6, b_5, a_7, b_7\}, \  
\{a_2, a_3, b_2, a_4, a_5, b_4, a_6, b_3, a_7, b_5, b_6, b_7\}, \  
\{a_2, a_3, b_2, a_4, a_5, a_6, b_4, b_3, b_5, b_6, a_7, b_7\}, \\  
& \{a_2, a_3, b_2, a_4, a_5, a_6, b_4, b_3, b_6, b_5, a_7, b_7\}, \  
\{a_2, a_3, b_2, a_4, a_5, a_6, b_4, b_3, a_7, b_5, b_6, b_7\}, \  
\{a_2, a_3, b_2, a_4, a_5, b_5, a_6, b_3, b_4, b_6, a_7, b_7\}, \\  
& \{a_2, a_3, b_2, a_4, a_5, a_6, b_5, b_3, b_4, b_6, a_7, b_7\}, \  
\{a_2, a_3, b_2, a_4, a_5, a_6, b_6, b_3, b_4, b_5, a_7, b_7\}, \  
\{a_2, a_3, b_2, a_4, a_5, a_6, a_7, b_3, b_4, b_5, b_6, b_7\}, \\  
& \{a_2, a_3, b_2, a_4, a_5, a_6, a_7, b_3, b_4, b_6, b_5, b_7\}, \  
\{a_2, a_3, b_2, a_4, a_5, b_5, a_6, b_3, b_6, b_4, a_7, b_7\}, \  
\{a_2, a_3, b_2, a_4, a_5, b_5, a_6, b_3, a_7, b_4, b_6, b_7\}, \\  
& \{a_2, a_3, b_2, a_4, a_5, a_6, b_5, b_3, b_6, b_4, a_7, b_7\}, \  
\{a_2, a_3, b_2, a_4, a_5, a_6, b_5, b_3, a_7, b_4, b_6, b_7\}, \  
\{a_2, a_3, b_2, a_4, a_5, a_6, b_6, b_3, b_5, b_4, a_7, b_7\}, \\  
& \{a_2, a_3, b_2, a_4, a_5, a_6, a_7, b_3, b_5, b_4, b_6, b_7\}, \  
\{a_2, a_3, b_2, a_4, a_5, a_6, a_7, b_3, b_5, b_6, b_4, b_7\}, \  
\{a_2, a_3, b_2, a_4, a_5, a_6, a_7, b_3, b_6, b_4, b_5, b_7\}, \\  
& \{a_2, a_3, b_2, a_4, a_5, a_6, a_7, b_3, b_6, b_5, b_4, b_7\}, \  
\{a_2, a_3, b_2, a_4, b_4, a_5, a_6, b_5, b_3, b_6, a_7, b_7\}, \  
\{a_2, a_3, b_2, a_4, b_4, a_5, a_6, a_7, b_3, b_5, b_6, b_7\}, \\  
& \{a_2, a_3, b_2, a_4, a_5, b_4, a_6, b_5, b_3, b_6, a_7, b_7\}, \  
\{a_2, a_3, b_2, a_4, a_5, b_4, a_6, a_7, b_3, b_5, b_6, b_7\}, \  
\{a_2, a_3, b_2, a_4, a_5, b_4, a_6, a_7, b_3, b_6, b_5, b_7\}, \\  
& \{a_2, a_3, b_2, a_4, a_5, a_6, b_4, b_5, b_3, b_6, a_7, b_7\}, \  
\{a_2, a_3, b_2, a_4, a_5, a_6, b_4, b_6, b_3, b_5, a_7, b_7\}, \  
\{a_2, a_3, b_2, a_4, a_5, a_6, b_4, a_7, b_3, b_5, b_6, b_7\}, \\  
& \{a_2, a_3, b_2, a_4, a_5, a_6, b_4, a_7, b_3, b_6, b_5, b_7\}, \  
\{a_2, a_3, b_2, a_4, a_5, a_6, b_5, a_7, b_3, b_4, b_6, b_7\}, \  
\{a_2, a_3, b_2, a_4, b_4, a_5, b_5, a_6, b_6, b_3, a_7, b_7\}, \\  
& \{a_2, a_3, b_2, a_4, b_4, a_5, b_5, a_6, a_7, b_3, b_6, b_7\}, \  
\{a_2, a_3, b_2, a_4, b_4, a_5, a_6, b_5, a_7, b_3, b_6, b_7\}, \  
\{a_2, a_3, b_2, a_4, a_5, b_4, b_5, a_6, a_7, b_3, b_6, b_7\}, \\  
& \{a_2, a_3, b_2, a_4, a_5, b_4, a_6, b_5, a_7, b_3, b_6, b_7\}, \  
\{a_2, a_3, b_3, b_2, a_4, b_4, a_5, a_6, a_7, b_5, b_6, b_7\}, \  
\{a_2, a_3, b_3, b_2, a_4, a_5, b_4, a_6, a_7, b_5, b_6, b_7\}, \\  
& \{a_2, a_3, b_3, b_2, a_4, a_5, a_6, b_4, b_5, b_6, a_7, b_7\}, \  
\{a_2, a_3, b_3, b_2, a_4, a_5, a_6, b_4, b_6, b_5, a_7, b_7\}, \  
\{a_2, a_3, b_3, b_2, a_4, a_5, a_6, b_4, a_7, b_5, b_6, b_7\}, \\  
& \{a_2, a_3, b_3, b_2, a_4, a_5, a_6, b_5, b_4, b_6, a_7, b_7\}, \  
\{a_2, a_3, a_4, b_2, b_3, b_4, a_5, a_6, a_7, b_5, b_6, b_7\}, \  
\{a_2, a_3, a_4, b_2, b_3, a_5, b_4, a_6, b_6, b_5, a_7, b_7\}, 
\end{align*}

\begin{align*} 
&\text{$103$--$228$}\\
& \{a_2, a_3, a_4, b_2, b_3, a_5, b_4, a_6, a_7, b_5, b_6, b_7\}, \  
\{a_2, a_3, a_4, b_2, b_3, a_5, a_6, b_4, b_6, b_5, a_7, b_7\}, \  
\{a_2, a_3, a_4, b_2, b_3, a_5, a_6, b_4, a_7, b_5, b_6, b_7\}, \\  
& \{a_2, a_3, a_4, b_2, b_3, a_5, a_6, b_5, b_4, b_6, a_7, b_7\}, \  
\{a_2, a_3, a_4, b_2, b_3, a_5, a_6, a_7, b_4, b_5, b_6, b_7\}, \  
\{a_2, a_3, a_4, b_2, b_3, a_5, b_5, a_6, b_6, b_4, a_7, b_7\}, \\  
& \{a_2, a_3, a_4, b_2, b_3, a_5, b_5, a_6, a_7, b_4, b_6, b_7\}, \  
\{a_2, a_3, a_4, b_2, b_3, a_5, a_6, b_5, a_7, b_4, b_6, b_7\}, \  
\{a_2, a_3, a_4, b_2, b_4, b_3, a_5, a_6, a_7, b_5, b_6, b_7\}, \\  
& \{a_2, a_3, a_4, b_2, a_5, b_3, b_4, a_6, b_6, b_5, a_7, b_7\}, \  
\{a_2, a_3, a_4, b_2, a_5, b_3, b_4, a_6, a_7, b_5, b_6, b_7\}, \  
\{a_2, a_3, a_4, b_2, a_5, b_3, a_6, b_4, b_6, b_5, a_7, b_7\}, \\  
& \{a_2, a_3, a_4, b_2, a_5, b_3, a_6, b_4, a_7, b_5, b_6, b_7\}, \  
\{a_2, a_3, a_4, b_2, a_5, b_3, a_6, b_5, b_4, b_6, a_7, b_7\}, \  
\{a_2, a_3, a_4, b_2, a_5, b_3, a_6, b_6, b_4, b_5, a_7, b_7\}, \\  
& \{a_2, a_3, a_4, b_2, a_5, b_3, a_6, a_7, b_4, b_5, b_6, b_7\}, \  
\{a_2, a_3, a_4, b_2, a_5, b_3, a_6, a_7, b_4, b_6, b_5, b_7\}, \  
\{a_2, a_3, a_4, b_2, a_5, b_3, b_5, a_6, b_6, b_4, a_7, b_7\}, \\ 
& \{a_2, a_3, a_4, b_2, a_5, b_3, b_5, a_6, a_7, b_4, b_6, b_7\}, \  
\{a_2, a_3, a_4, b_2, a_5, b_3, a_6, b_5, b_6, b_4, a_7, b_7\}, \  
\{a_2, a_3, a_4, b_2, a_5, b_3, a_6, b_5, a_7, b_4, b_6, b_7\}, \\  
& \{a_2, a_3, a_4, b_2, b_4, a_5, b_3, a_6, b_6, b_5, a_7, b_7\}, \  
\{a_2, a_3, a_4, b_2, b_4, a_5, b_3, a_6, a_7, b_5, b_6, b_7\}, \  
\{a_2, a_3, a_4, b_2, a_5, b_4, b_3, a_6, b_6, b_5, a_7, b_7\}, \\  
& \{a_2, a_3, a_4, b_2, a_5, b_4, b_3, a_6, a_7, b_5, b_6, b_7\}, \  
\{a_2, a_3, a_4, b_2, a_5, a_6, b_3, b_4, b_5, b_6, a_7, b_7\}, \  
\{a_2, a_3, a_4, b_2, a_5, a_6, b_3, b_4, b_6, b_5, a_7, b_7\}, \\  
& \{a_2, a_3, a_4, b_2, a_5, a_6, b_3, b_4, a_7, b_5, b_6, b_7\}, \  
\{a_2, a_3, a_4, b_2, a_5, a_6, b_3, b_5, b_4, b_6, a_7, b_7\}, \  
\{a_2, a_3, a_4, b_2, a_5, a_6, b_3, b_6, b_4, b_5, a_7, b_7\}, \\  
& \{a_2, a_3, a_4, b_2, a_5, a_6, b_3, a_7, b_4, b_5, b_6, b_7\}, \  
\{a_2, a_3, a_4, b_2, a_5, a_6, b_3, a_7, b_4, b_6, b_5, b_7\}, \  
\{a_2, a_3, a_4, b_2, a_5, b_5, b_3, a_6, b_6, b_4, a_7, b_7\}, \\  
& \{a_2, a_3, a_4, b_2, a_5, b_5, b_3, a_6, a_7, b_4, b_6, b_7\}, \  
\{a_2, a_3, a_4, b_2, a_5, a_6, b_3, b_5, b_6, b_4, a_7, b_7\}, \  
\{a_2, a_3, a_4, b_2, a_5, a_6, b_3, b_5, a_7, b_4, b_6, b_7\}, \\  
& \{a_2, a_3, a_4, b_2, a_5, a_6, b_3, b_6, b_5, b_4, a_7, b_7\}, \  
\{a_2, a_3, a_4, b_2, a_5, a_6, b_3, a_7, b_5, b_4, b_6, b_7\}, \  
\{a_2, a_3, a_4, b_2, a_5, a_6, b_3, a_7, b_5, b_6, b_4, b_7\}, \\  
& \{a_2, a_3, a_4, b_2, b_4, a_5, a_6, b_3, b_6, b_5, a_7, b_7\}, \  
\{a_2, a_3, a_4, b_2, b_4, a_5, a_6, b_3, a_7, b_5, b_6, b_7\}, \  
\{a_2, a_3, a_4, b_2, a_5, b_4, a_6, b_3, b_6, b_5, a_7, b_7\}, \\  
& \{a_2, a_3, a_4, b_2, a_5, b_4, a_6, b_3, a_7, b_5, b_6, b_7\}, \  
\{a_2, a_3, a_4, b_2, a_5, a_6, b_4, b_3, b_5, b_6, a_7, b_7\}, \  
\{a_2, a_3, a_4, b_2, a_5, a_6, b_4, b_3, b_6, b_5, a_7, b_7\}, \\  
& \{a_2, a_3, a_4, b_2, a_5, a_6, b_4, b_3, a_7, b_5, b_6, b_7\}, \  
\{a_2, a_3, a_4, b_2, a_5, a_6, b_5, b_3, b_4, b_6, a_7, b_7\}, \  
\{a_2, a_3, a_4, b_2, a_5, a_6, b_6, b_3, b_4, b_5, a_7, b_7\}, \\  
& \{a_2, a_3, a_4, b_2, a_5, a_6, a_7, b_3, b_4, b_5, b_6, b_7\}, \  
\{a_2, a_3, a_4, b_2, a_5, a_6, a_7, b_3, b_4, b_6, b_5, b_7\}, \  
\{a_2, a_3, a_4, b_2, a_5, b_5, a_6, b_3, b_6, b_4, a_7, b_7\}, \\  
& \{a_2, a_3, a_4, b_2, a_5, b_5, a_6, b_3, a_7, b_4, b_6, b_7\}, \  
\{a_2, a_3, a_4, b_2, a_5, a_6, b_5, b_3, b_6, b_4, a_7, b_7\}, \  
\{a_2, a_3, a_4, b_2, a_5, a_6, b_5, b_3, a_7, b_4, b_6, b_7\}, \\  
& \{a_2, a_3, a_4, b_2, a_5, a_6, b_6, b_3, b_5, b_4, a_7, b_7\}, \  
\{a_2, a_3, a_4, b_2, a_5, a_6, a_7, b_3, b_5, b_4, b_6, b_7\}, \  
\{a_2, a_3, a_4, b_2, a_5, a_6, a_7, b_3, b_5, b_6, b_4, b_7\}, \\  
& \{a_2, a_3, a_4, b_2, a_5, a_6, b_6, b_3, a_7, b_4, b_5, b_7\}, \  
\{a_2, a_3, a_4, b_2, a_5, a_6, a_7, b_3, b_6, b_4, b_5, b_7\}, \  
\{a_2, a_3, a_4, b_2, a_5, a_6, a_7, b_3, b_6, b_5, b_4, b_7\}, \\  
& \{a_2, a_3, a_4, b_2, b_4, a_5, a_6, a_7, b_3, b_5, b_6, b_7\}, \  
\{a_2, a_3, a_4, b_2, a_5, b_4, a_6, a_7, b_3, b_5, b_6, b_7\}, \  
\{a_2, a_3, a_4, b_2, a_5, b_4, a_6, a_7, b_3, b_6, b_5, b_7\}, \\  
& \{a_2, a_3, a_4, b_2, a_5, a_6, b_4, b_5, b_3, b_6, a_7, b_7\}, \  
\{a_2, a_3, a_4, b_2, a_5, a_6, b_4, b_6, b_3, b_5, a_7, b_7\}, \  
\{a_2, a_3, a_4, b_2, a_5, a_6, b_4, a_7, b_3, b_5, b_6, b_7\}, \\  
& \{a_2, a_3, a_4, b_2, a_5, a_6, b_4, a_7, b_3, b_6, b_5, b_7\}, \  
\{a_2, a_3, a_4, b_2, a_5, b_5, a_6, b_6, b_3, b_4, a_7, b_7\}, \  
\{a_2, a_3, a_4, b_2, a_5, b_5, a_6, a_7, b_3, b_4, b_6, b_7\}, \\  
& \{a_2, a_3, a_4, b_2, a_5, b_5, a_6, a_7, b_3, b_6, b_4, b_7\}, \  
\{a_2, a_3, a_4, b_2, a_5, a_6, b_5, b_6, b_3, b_4, a_7, b_7\}, \  
\{a_2, a_3, a_4, b_2, a_5, a_6, b_5, a_7, b_3, b_4, b_6, b_7\}, \\  
& \{a_2, a_3, a_4, b_2, a_5, a_6, b_5, a_7, b_3, b_6, b_4, b_7\}, \  
\{a_2, a_3, a_4, b_2, a_5, a_6, a_7, b_6, b_3, b_4, b_5, b_7\}, \  
\{a_2, a_3, a_4, b_2, b_4, a_5, b_5, a_6, b_6, b_3, a_7, b_7\}, \\  
& \{a_2, a_3, a_4, b_2, b_4, a_5, b_5, a_6, a_7, b_3, b_6, b_7\}, \  
\{a_2, a_3, a_4, b_2, b_4, a_5, a_6, b_5, a_7, b_3, b_6, b_7\}, \  
\{a_2, a_3, a_4, b_2, a_5, b_4, b_5, a_6, b_6, b_3, a_7, b_7\}, \\  
& \{a_2, a_3, a_4, b_2, a_5, b_4, b_5, a_6, a_7, b_3, b_6, b_7\}, \  
\{a_2, a_3, a_4, b_2, a_5, b_4, a_6, b_5, b_6, b_3, a_7, b_7\}, \  
\{a_2, a_3, a_4, b_2, a_5, b_4, a_6, b_5, a_7, b_3, b_6, b_7\}, \\  
& \{a_2, a_3, b_3, a_4, b_2, a_5, b_4, a_6, a_7, b_5, b_6, b_7\}, \  
\{a_2, a_3, b_3, a_4, b_2, a_5, a_6, b_4, b_6, b_5, a_7, b_7\}, \  
\{a_2, a_3, b_3, a_4, b_2, a_5, a_6, b_4, a_7, b_5, b_6, b_7\}, \\  
& \{a_2, a_3, b_3, a_4, b_2, a_5, a_6, a_7, b_4, b_5, b_6, b_7\}, \  
\{a_2, a_3, b_3, a_4, b_2, a_5, b_5, a_6, a_7, b_4, b_6, b_7\}, \  
\{a_2, a_3, b_3, a_4, b_2, a_5, a_6, b_5, a_7, b_4, b_6, b_7\}, \\  
& \{a_2, a_3, a_4, b_3, b_2, b_4, a_5, a_6, a_7, b_5, b_6, b_7\}, \  
\{a_2, a_3, a_4, b_3, b_2, a_5, b_4, a_6, a_7, b_5, b_6, b_7\}, \  
\{a_2, a_3, a_4, b_3, b_2, a_5, a_6, b_4, a_7, b_5, b_6, b_7\}, \\  
& \{a_2, a_3, a_4, b_3, b_2, a_5, a_6, a_7, b_4, b_5, b_6, b_7\}, \  
\{a_2, a_3, a_4, b_3, b_2, a_5, b_5, a_6, b_6, b_4, a_7, b_7\}, \  
\{a_2, a_3, a_4, b_3, b_2, a_5, b_5, a_6, a_7, b_4, b_6, b_7\}, \\  
& \{a_2, a_3, a_4, b_3, b_2, a_5, a_6, b_5, a_7, b_4, b_6, b_7\}, \  
\{a_2, a_3, a_4, a_5, b_2, b_3, b_4, a_6, b_6, b_5, a_7, b_7\}, \  
\{a_2, a_3, a_4, a_5, b_2, b_3, b_4, a_6, a_7, b_5, b_6, b_7\}, \\  
& \{a_2, a_3, a_4, a_5, b_2, b_3, a_6, b_4, b_6, b_5, a_7, b_7\}, \  
\{a_2, a_3, a_4, a_5, b_2, b_3, a_6, b_4, a_7, b_5, b_6, b_7\}, \  
\{a_2, a_3, a_4, a_5, b_2, b_3, a_6, b_6, b_4, b_5, a_7, b_7\}, \\  
& \{a_2, a_3, a_4, a_5, b_2, b_3, a_6, a_7, b_4, b_5, b_6, b_7\}, \  
\{a_2, a_3, a_4, a_5, b_2, b_3, a_6, a_7, b_4, b_6, b_5, b_7\}, \  
\{a_2, a_3, a_4, a_5, b_2, b_3, b_5, a_6, b_6, b_4, a_7, b_7\}, \\  
& \{a_2, a_3, a_4, a_5, b_2, b_3, b_5, a_6, a_7, b_4, b_6, b_7\}, \  
\{a_2, a_3, a_4, a_5, b_2, b_3, a_6, b_5, b_6, b_4, a_7, b_7\}, \  
\{a_2, a_3, a_4, a_5, b_2, b_3, a_6, b_5, a_7, b_4, b_6, b_7\}, \\  
& \{a_2, a_3, a_4, a_5, b_2, b_4, b_3, a_6, b_6, b_5, a_7, b_7\}, \  
\{a_2, a_3, a_4, a_5, b_2, b_4, b_3, a_6, a_7, b_5, b_6, b_7\}, \  
\{a_2, a_3, a_4, a_5, b_2, a_6, b_3, b_4, b_6, b_5, a_7, b_7\}, \\  
& \{a_2, a_3, a_4, a_5, b_2, a_6, b_3, b_4, a_7, b_5, b_6, b_7\}, \  
\{a_2, a_3, a_4, a_5, b_2, a_6, b_3, b_5, b_4, b_6, a_7, b_7\}, \  
\{a_2, a_3, a_4, a_5, b_2, a_6, b_3, b_6, b_4, b_5, a_7, b_7\}, \\  
& \{a_2, a_3, a_4, a_5, b_2, a_6, b_3, a_7, b_4, b_5, b_6, b_7\}, \  
\{a_2, a_3, a_4, a_5, b_2, a_6, b_3, a_7, b_4, b_6, b_5, b_7\}, \  
\{a_2, a_3, a_4, a_5, b_2, b_5, b_3, a_6, b_6, b_4, a_7, b_7\}, \\  
& \{a_2, a_3, a_4, a_5, b_2, b_5, b_3, a_6, a_7, b_4, b_6, b_7\}, \  
\{a_2, a_3, a_4, a_5, b_2, a_6, b_3, b_5, b_6, b_4, a_7, b_7\}, \  
\{a_2, a_3, a_4, a_5, b_2, a_6, b_3, b_5, a_7, b_4, b_6, b_7\}, \\  
& \{a_2, a_3, a_4, a_5, b_2, a_6, b_3, b_6, b_5, b_4, a_7, b_7\}, \  
\{a_2, a_3, a_4, a_5, b_2, a_6, b_3, a_7, b_5, b_4, b_6, b_7\}, \  
\{a_2, a_3, a_4, a_5, b_2, a_6, b_3, a_7, b_5, b_6, b_4, b_7\}, \\  
& \{a_2, a_3, a_4, a_5, b_2, b_4, a_6, b_3, b_6, b_5, a_7, b_7\}, \  
\{a_2, a_3, a_4, a_5, b_2, b_4, a_6, b_3, a_7, b_5, b_6, b_7\}, \  
\{a_2, a_3, a_4, a_5, b_2, a_6, b_4, b_3, b_6, b_5, a_7, b_7\}, \\  
& \{a_2, a_3, a_4, a_5, b_2, a_6, b_4, b_3, a_7, b_5, b_6, b_7\}, \  
\{a_2, a_3, a_4, a_5, b_2, a_6, b_6, b_3, b_4, b_5, a_7, b_7\}, \  
\{a_2, a_3, a_4, a_5, b_2, a_6, a_7, b_3, b_4, b_5, b_6, b_7\}, 
\end{align*}

\begin{align*} 
&\text{$229$--$354$}\\
& \{a_2, a_3, a_4, a_5, b_2, a_6, a_7, b_3, b_4, b_6, b_5, b_7\}, \  
\{a_2, a_3, a_4, a_5, b_2, b_5, a_6, b_3, b_6, b_4, a_7, b_7\}, \  
\{a_2, a_3, a_4, a_5, b_2, b_5, a_6, b_3, a_7, b_4, b_6, b_7\}, \\  
& \{a_2, a_3, a_4, a_5, b_2, a_6, b_5, b_3, b_6, b_4, a_7, b_7\}, \  
\{a_2, a_3, a_4, a_5, b_2, a_6, b_5, b_3, a_7, b_4, b_6, b_7\}, \  
\{a_2, a_3, a_4, a_5, b_2, a_6, b_6, b_3, b_5, b_4, a_7, b_7\}, \\  
& \{a_2, a_3, a_4, a_5, b_2, a_6, a_7, b_3, b_5, b_4, b_6, b_7\}, \  
\{a_2, a_3, a_4, a_5, b_2, a_6, a_7, b_3, b_5, b_6, b_4, b_7\}, \  
\{a_2, a_3, a_4, a_5, b_2, a_6, b_6, b_3, a_7, b_4, b_5, b_7\}, \\  
& \{a_2, a_3, a_4, a_5, b_2, a_6, a_7, b_3, b_6, b_4, b_5, b_7\}, \  
\{a_2, a_3, a_4, a_5, b_2, a_6, a_7, b_3, b_6, b_5, b_4, b_7\}, \  
\{a_2, a_3, a_4, a_5, b_2, b_4, a_6, b_6, b_3, b_5, a_7, b_7\}, \\  
& \{a_2, a_3, a_4, a_5, b_2, b_4, a_6, a_7, b_3, b_5, b_6, b_7\}, \  
\{a_2, a_3, a_4, a_5, b_2, b_4, a_6, a_7, b_3, b_6, b_5, b_7\}, \  
\{a_2, a_3, a_4, a_5, b_2, a_6, b_4, b_6, b_3, b_5, a_7, b_7\}, \\  
& \{a_2, a_3, a_4, a_5, b_2, a_6, b_4, a_7, b_3, b_5, b_6, b_7\}, \  
\{a_2, a_3, a_4, a_5, b_2, a_6, b_4, a_7, b_3, b_6, b_5, b_7\}, \  
\{a_2, a_3, a_4, a_5, b_2, a_6, a_7, b_4, b_3, b_5, b_6, b_7\}, \\  
& \{a_2, a_3, a_4, a_5, b_2, a_6, a_7, b_4, b_3, b_6, b_5, b_7\}, \  
\{a_2, a_3, a_4, a_5, b_2, b_5, a_6, b_6, b_3, b_4, a_7, b_7\}, \  
\{a_2, a_3, a_4, a_5, b_2, b_5, a_6, a_7, b_3, b_4, b_6, b_7\}, \\  
& \{a_2, a_3, a_4, a_5, b_2, b_5, a_6, a_7, b_3, b_6, b_4, b_7\}, \  
\{a_2, a_3, a_4, a_5, b_2, a_6, b_5, b_6, b_3, b_4, a_7, b_7\}, \  
\{a_2, a_3, a_4, a_5, b_2, a_6, b_5, a_7, b_3, b_4, b_6, b_7\}, \\  
& \{a_2, a_3, a_4, a_5, b_2, a_6, b_5, a_7, b_3, b_6, b_4, b_7\}, \  
\{a_2, a_3, a_4, a_5, b_2, a_6, b_6, a_7, b_3, b_4, b_5, b_7\}, \  
\{a_2, a_3, a_4, a_5, b_2, a_6, b_6, a_7, b_3, b_5, b_4, b_7\}, \\  
& \{a_2, a_3, a_4, a_5, b_2, a_6, a_7, b_6, b_3, b_4, b_5, b_7\}, \  
\{a_2, a_3, a_4, a_5, b_2, a_6, a_7, b_6, b_3, b_5, b_4, b_7\}, \  
\{a_2, a_3, a_4, a_5, b_2, b_4, b_5, a_6, b_6, b_3, a_7, b_7\}, \\  
& \{a_2, a_3, a_4, a_5, b_2, b_4, b_5, a_6, a_7, b_3, b_6, b_7\}, \  
\{a_2, a_3, a_4, a_5, b_2, b_4, a_6, b_5, b_6, b_3, a_7, b_7\}, \  
\{a_2, a_3, a_4, a_5, b_2, b_4, a_6, b_5, a_7, b_3, b_6, b_7\}, \\  
& \{a_2, a_3, a_4, a_5, b_2, b_5, b_4, a_6, b_6, b_3, a_7, b_7\}, \  
\{a_2, a_3, a_4, a_5, b_2, b_5, b_4, a_6, a_7, b_3, b_6, b_7\}, \  
\{a_2, a_3, a_4, a_5, b_2, a_6, b_4, b_5, b_6, b_3, a_7, b_7\}, \\  
& \{a_2, a_3, a_4, a_5, b_2, a_6, b_4, b_5, a_7, b_3, b_6, b_7\}, \  
\{a_2, a_3, a_4, a_5, b_2, a_6, b_4, b_6, b_5, b_3, a_7, b_7\}, \  
\{a_2, a_3, a_4, a_5, b_2, a_6, b_4, a_7, b_5, b_3, b_6, b_7\}, \\  
& \{a_2, a_3, a_4, a_5, b_2, a_6, b_4, a_7, b_5, b_6, b_3, b_7\}, \  
\{a_2, a_3, b_3, a_4, a_5, b_2, b_4, a_6, a_7, b_5, b_6, b_7\}, \  
\{a_2, a_3, b_3, a_4, a_5, b_2, a_6, b_4, b_6, b_5, a_7, b_7\}, \\  
& \{a_2, a_3, b_3, a_4, a_5, b_2, a_6, b_4, a_7, b_5, b_6, b_7\}, \  
\{a_2, a_3, b_3, a_4, a_5, b_2, a_6, a_7, b_4, b_5, b_6, b_7\}, \  
\{a_2, a_3, b_3, a_4, a_5, b_2, a_6, a_7, b_4, b_6, b_5, b_7\}, \\  
& \{a_2, a_3, b_3, a_4, a_5, b_2, b_5, a_6, b_6, b_4, a_7, b_7\}, \  
\{a_2, a_3, b_3, a_4, a_5, b_2, b_5, a_6, a_7, b_4, b_6, b_7\}, \  
\{a_2, a_3, b_3, a_4, a_5, b_2, a_6, b_5, b_6, b_4, a_7, b_7\}, \\  
& \{a_2, a_3, b_3, a_4, a_5, b_2, a_6, b_5, a_7, b_4, b_6, b_7\}, \  
\{a_2, a_3, a_4, b_3, a_5, b_2, b_4, a_6, a_7, b_5, b_6, b_7\}, \  
\{a_2, a_3, a_4, b_3, a_5, b_2, a_6, b_4, a_7, b_5, b_6, b_7\}, \\  
& \{a_2, a_3, a_4, b_3, a_5, b_2, a_6, a_7, b_4, b_5, b_6, b_7\}, \  
\{a_2, a_3, a_4, b_3, a_5, b_2, a_6, a_7, b_4, b_6, b_5, b_7\}, \  
\{a_2, a_3, a_4, b_3, a_5, b_2, b_5, a_6, b_6, b_4, a_7, b_7\}, \\  
& \{a_2, a_3, a_4, b_3, a_5, b_2, b_5, a_6, a_7, b_4, b_6, b_7\}, \  
\{a_2, a_3, a_4, b_3, a_5, b_2, a_6, b_5, b_6, b_4, a_7, b_7\}, \  
\{a_2, a_3, a_4, b_3, a_5, b_2, a_6, b_5, a_7, b_4, b_6, b_7\}, \\  
& \{a_2, a_3, a_4, a_5, b_3, b_2, b_4, a_6, b_6, b_5, a_7, b_7\}, \  
\{a_2, a_3, a_4, a_5, b_3, b_2, b_4, a_6, a_7, b_5, b_6, b_7\}, \  
\{a_2, a_3, a_4, a_5, b_3, b_2, a_6, b_4, b_6, b_5, a_7, b_7\}, \\  
& \{a_2, a_3, a_4, a_5, b_3, b_2, a_6, b_4, a_7, b_5, b_6, b_7\}, \  
\{a_2, a_3, a_4, a_5, b_3, b_2, a_6, a_7, b_4, b_5, b_6, b_7\}, \  
\{a_2, a_3, a_4, a_5, b_3, b_2, a_6, a_7, b_4, b_6, b_5, b_7\}, \\  
& \{a_2, a_3, a_4, a_5, b_3, b_2, b_5, a_6, b_6, b_4, a_7, b_7\}, \  
\{a_2, a_3, a_4, a_5, b_3, b_2, b_5, a_6, a_7, b_4, b_6, b_7\}, \  
\{a_2, a_3, a_4, a_5, b_3, b_2, a_6, b_5, b_6, b_4, a_7, b_7\}, \\  
& \{a_2, a_3, a_4, a_5, b_3, b_2, a_6, b_5, a_7, b_4, b_6, b_7\}, \  
\{a_2, a_3, a_4, a_5, b_4, b_2, b_3, a_6, b_6, b_5, a_7, b_7\}, \  
\{a_2, a_3, a_4, a_5, b_4, b_2, b_3, a_6, a_7, b_5, b_6, b_7\}, \\  
& \{a_2, a_3, a_4, a_5, a_6, b_2, b_3, b_4, b_6, b_5, a_7, b_7\}, \  
\{a_2, a_3, a_4, a_5, a_6, b_2, b_3, b_4, a_7, b_5, b_6, b_7\}, \  
\{a_2, a_3, a_4, a_5, a_6, b_2, b_3, b_5, b_4, b_6, a_7, b_7\}, \\  
& \{a_2, a_3, a_4, a_5, a_6, b_2, b_3, b_6, b_4, b_5, a_7, b_7\}, \  
\{a_2, a_3, a_4, a_5, a_6, b_2, b_3, a_7, b_4, b_5, b_6, b_7\}, \  
\{a_2, a_3, a_4, a_5, a_6, b_2, b_3, a_7, b_4, b_6, b_5, b_7\}, \\  
& \{a_2, a_3, a_4, a_5, b_5, b_2, b_3, a_6, a_7, b_4, b_6, b_7\}, \  
\{a_2, a_3, a_4, a_5, a_6, b_2, b_3, b_5, b_6, b_4, a_7, b_7\}, \  
\{a_2, a_3, a_4, a_5, a_6, b_2, b_3, b_5, a_7, b_4, b_6, b_7\}, \\  
& \{a_2, a_3, a_4, a_5, a_6, b_2, b_3, b_6, b_5, b_4, a_7, b_7\}, \  
\{a_2, a_3, a_4, a_5, a_6, b_2, b_3, a_7, b_5, b_4, b_6, b_7\}, \  
\{a_2, a_3, a_4, a_5, a_6, b_2, b_3, a_7, b_5, b_6, b_4, b_7\}, \\  
& \{a_2, a_3, a_4, a_5, b_4, b_2, a_6, b_3, b_6, b_5, a_7, b_7\}, \  
\{a_2, a_3, a_4, a_5, b_4, b_2, a_6, b_3, a_7, b_5, b_6, b_7\}, \  
\{a_2, a_3, a_4, a_5, a_6, b_2, b_4, b_3, b_6, b_5, a_7, b_7\}, \\  
& \{a_2, a_3, a_4, a_5, a_6, b_2, b_4, b_3, a_7, b_5, b_6, b_7\}, \  
\{a_2, a_3, a_4, a_5, a_6, b_2, b_6, b_3, b_4, b_5, a_7, b_7\}, \  
\{a_2, a_3, a_4, a_5, a_6, b_2, a_7, b_3, b_4, b_5, b_6, b_7\}, \\  
& \{a_2, a_3, a_4, a_5, a_6, b_2, a_7, b_3, b_4, b_6, b_5, b_7\}, \  
\{a_2, a_3, a_4, a_5, b_5, b_2, a_6, b_3, b_6, b_4, a_7, b_7\}, \  
\{a_2, a_3, a_4, a_5, b_5, b_2, a_6, b_3, a_7, b_4, b_6, b_7\}, \\  
& \{a_2, a_3, a_4, a_5, a_6, b_2, b_5, b_3, b_6, b_4, a_7, b_7\}, \  
\{a_2, a_3, a_4, a_5, a_6, b_2, b_5, b_3, a_7, b_4, b_6, b_7\}, \  
\{a_2, a_3, a_4, a_5, a_6, b_2, b_6, b_3, b_5, b_4, a_7, b_7\}, \\  
& \{a_2, a_3, a_4, a_5, a_6, b_2, a_7, b_3, b_5, b_4, b_6, b_7\}, \  
\{a_2, a_3, a_4, a_5, a_6, b_2, a_7, b_3, b_5, b_6, b_4, b_7\}, \  
\{a_2, a_3, a_4, a_5, a_6, b_2, b_6, b_3, a_7, b_4, b_5, b_7\}, \\  
& \{a_2, a_3, a_4, a_5, a_6, b_2, a_7, b_3, b_6, b_4, b_5, b_7\}, \  
\{a_2, a_3, a_4, a_5, a_6, b_2, a_7, b_3, b_6, b_5, b_4, b_7\}, \  
\{a_2, a_3, a_4, b_4, a_5, b_2, a_6, a_7, b_3, b_5, b_6, b_7\}, \\  
& \{a_2, a_3, a_4, b_4, a_5, b_2, a_6, a_7, b_3, b_6, b_5, b_7\}, \  
\{a_2, a_3, a_4, a_5, b_4, b_2, a_6, b_6, b_3, b_5, a_7, b_7\}, \  
\{a_2, a_3, a_4, a_5, b_4, b_2, a_6, a_7, b_3, b_5, b_6, b_7\}, \\  
& \{a_2, a_3, a_4, a_5, b_4, b_2, a_6, a_7, b_3, b_6, b_5, b_7\}, \  
\{a_2, a_3, a_4, a_5, a_6, b_2, b_4, b_6, b_3, b_5, a_7, b_7\}, \  
\{a_2, a_3, a_4, a_5, a_6, b_2, b_4, a_7, b_3, b_5, b_6, b_7\}, \\  
& \{a_2, a_3, a_4, a_5, a_6, b_2, b_4, a_7, b_3, b_6, b_5, b_7\}, \  
\{a_2, a_3, a_4, a_5, a_6, b_2, b_6, b_4, b_3, b_5, a_7, b_7\}, \  
\{a_2, a_3, a_4, a_5, a_6, b_2, a_7, b_4, b_3, b_5, b_6, b_7\}, \\  
& \{a_2, a_3, a_4, a_5, a_6, b_2, a_7, b_4, b_3, b_6, b_5, b_7\}, \  
\{a_2, a_3, a_4, a_5, a_6, b_2, b_5, b_6, b_3, b_4, a_7, b_7\}, \  
\{a_2, a_3, a_4, a_5, a_6, b_2, b_5, a_7, b_3, b_4, b_6, b_7\}, \\  
& \{a_2, a_3, a_4, a_5, a_6, b_2, b_5, a_7, b_3, b_6, b_4, b_7\}, \  
\{a_2, a_3, a_4, a_5, a_6, b_2, b_6, b_5, b_3, b_4, a_7, b_7\}, \  
\{a_2, a_3, a_4, a_5, a_6, b_2, a_7, b_5, b_3, b_4, b_6, b_7\}, \\  
& \{a_2, a_3, a_4, a_5, a_6, b_2, a_7, b_5, b_3, b_6, b_4, b_7\}, \  
\{a_2, a_3, a_4, a_5, a_6, b_2, b_6, a_7, b_3, b_4, b_5, b_7\}, \  
\{a_2, a_3, a_4, a_5, a_6, b_2, b_6, a_7, b_3, b_5, b_4, b_7\}, \\
& \{a_2, a_3, a_4, a_5, a_6, b_2, a_7, b_6, b_3, b_4, b_5, b_7\}, \  
\{a_2, a_3, a_4, a_5, a_6, b_2, a_7, b_6, b_3, b_5, b_4, b_7\}, \  
\{a_2, a_3, a_4, a_5, a_6, b_2, a_7, b_7, b_3, b_4, b_5, b_6\}, \\  
& \{a_2, a_3, a_4, a_5, a_6, b_2, a_7, b_7, b_3, b_5, b_4, b_6\}, \  
\{a_2, a_3, a_4, b_4, a_5, b_2, b_5, a_6, a_7, b_3, b_6, b_7\}, \  
\{a_2, a_3, a_4, b_4, a_5, b_2, a_6, b_5, a_7, b_3, b_6, b_7\}, \\  
& \{a_2, a_3, a_4, a_5, b_4, b_2, b_5, a_6, b_6, b_3, a_7, b_7\}, \  
\{a_2, a_3, a_4, a_5, b_4, b_2, b_5, a_6, a_7, b_3, b_6, b_7\}, \  
\{a_2, a_3, a_4, a_5, b_4, b_2, a_6, b_5, b_6, b_3, a_7, b_7\}, 
\end{align*}

\begin{align*} 
&\text{$355$--$469$}\\
& \{a_2, a_3, a_4, a_5, b_4, b_2, a_6, b_5, a_7, b_3, b_6, b_7\}, \  
\{a_2, a_3, a_4, a_5, a_6, b_2, b_4, b_5, b_6, b_3, a_7, b_7\}, \  
\{a_2, a_3, a_4, a_5, a_6, b_2, b_4, b_5, a_7, b_3, b_6, b_7\}, \\  
& \{a_2, a_3, a_4, a_5, a_6, b_2, b_4, b_6, b_5, b_3, a_7, b_7\}, \  
\{a_2, a_3, a_4, a_5, a_6, b_2, b_4, a_7, b_5, b_3, b_6, b_7\}, \  
\{a_2, a_3, a_4, a_5, a_6, b_2, b_4, a_7, b_5, b_6, b_3, b_7\}, \\  
& \{a_2, a_3, a_4, a_5, a_6, b_2, b_5, b_4, b_6, b_3, a_7, b_7\}, \  
\{a_2, a_3, a_4, a_5, a_6, b_2, b_5, b_4, a_7, b_3, b_6, b_7\}, \  
\{a_2, a_3, a_4, a_5, a_6, b_2, b_6, b_4, b_5, b_3, a_7, b_7\}, \\  
& \{a_2, a_3, a_4, a_5, a_6, b_2, a_7, b_4, b_5, b_3, b_6, b_7\}, \  
\{a_2, a_3, a_4, a_5, a_6, b_2, a_7, b_4, b_5, b_6, b_3, b_7\}, \  
\{a_2, a_3, a_4, a_5, a_6, b_2, b_6, b_4, a_7, b_3, b_5, b_7\}, \\  
& \{a_2, a_3, a_4, a_5, a_6, b_2, a_7, b_4, b_6, b_3, b_5, b_7\}, \  
\{a_2, a_3, a_4, a_5, a_6, b_2, a_7, b_4, b_6, b_5, b_3, b_7\}, \  
\{a_2, a_3, b_3, a_4, b_4, a_5, b_2, a_6, a_7, b_5, b_6, b_7\}, \\  
& \{a_2, a_3, b_3, a_4, a_5, b_4, b_2, a_6, a_7, b_5, b_6, b_7\}, \  
\{a_2, a_3, b_3, a_4, a_5, a_6, b_2, b_4, a_7, b_5, b_6, b_7\}, \  
\{a_2, a_3, b_3, a_4, a_5, a_6, b_2, a_7, b_4, b_5, b_6, b_7\}, \\  
& \{a_2, a_3, b_3, a_4, a_5, a_6, b_2, a_7, b_4, b_6, b_5, b_7\}, \  
\{a_2, a_3, b_3, a_4, a_5, a_6, b_2, b_5, a_7, b_4, b_6, b_7\}, \  
\{a_2, a_3, b_3, a_4, a_5, a_6, b_2, a_7, b_5, b_4, b_6, b_7\}, \\  
& \{a_2, a_3, b_3, a_4, a_5, a_6, b_2, a_7, b_5, b_6, b_4, b_7\}, \  
\{a_2, a_3, a_4, b_3, b_4, a_5, b_2, a_6, a_7, b_5, b_6, b_7\}, \  
\{a_2, a_3, a_4, b_3, a_5, b_4, b_2, a_6, a_7, b_5, b_6, b_7\}, \\  
& \{a_2, a_3, a_4, b_3, a_5, a_6, b_2, b_4, a_7, b_5, b_6, b_7\}, \  
\{a_2, a_3, a_4, b_3, a_5, a_6, b_2, a_7, b_4, b_5, b_6, b_7\}, \  
\{a_2, a_3, a_4, b_3, a_5, a_6, b_2, a_7, b_4, b_6, b_5, b_7\}, \\  
& \{a_2, a_3, a_4, b_3, a_5, a_6, b_2, b_5, b_6, b_4, a_7, b_7\}, \  
\{a_2, a_3, a_4, b_3, a_5, a_6, b_2, b_5, a_7, b_4, b_6, b_7\}, \  
\{a_2, a_3, a_4, b_3, a_5, a_6, b_2, a_7, b_5, b_4, b_6, b_7\}, \\  
& \{a_2, a_3, a_4, b_3, a_5, a_6, b_2, a_7, b_5, b_6, b_4, b_7\}, \  
\{a_2, a_3, a_4, a_5, b_3, b_4, b_2, a_6, a_7, b_5, b_6, b_7\}, \  
\{a_2, a_3, a_4, a_5, b_3, a_6, b_2, b_4, a_7, b_5, b_6, b_7\}, \\  
& \{a_2, a_3, a_4, a_5, b_3, a_6, b_2, b_6, b_4, b_5, a_7, b_7\}, \  
\{a_2, a_3, a_4, a_5, b_3, a_6, b_2, a_7, b_4, b_5, b_6, b_7\}, \  
\{a_2, a_3, a_4, a_5, b_3, a_6, b_2, a_7, b_4, b_6, b_5, b_7\}, \\  
& \{a_2, a_3, a_4, a_5, b_3, b_5, b_2, a_6, a_7, b_4, b_6, b_7\}, \  
\{a_2, a_3, a_4, a_5, b_3, a_6, b_2, b_5, b_6, b_4, a_7, b_7\}, \  
\{a_2, a_3, a_4, a_5, b_3, a_6, b_2, b_5, a_7, b_4, b_6, b_7\}, \\  
& \{a_2, a_3, a_4, a_5, b_3, a_6, b_2, b_6, b_5, b_4, a_7, b_7\}, \  
\{a_2, a_3, a_4, a_5, b_3, a_6, b_2, a_7, b_5, b_4, b_6, b_7\}, \  
\{a_2, a_3, a_4, a_5, b_3, a_6, b_2, a_7, b_5, b_6, b_4, b_7\}, \\  
& \{a_2, a_3, a_4, a_5, a_6, b_3, b_2, b_4, a_7, b_5, b_6, b_7\}, \  
\{a_2, a_3, a_4, a_5, a_6, b_3, b_2, b_6, b_4, b_5, a_7, b_7\}, \  
\{a_2, a_3, a_4, a_5, a_6, b_3, b_2, a_7, b_4, b_5, b_6, b_7\}, \\  
& \{a_2, a_3, a_4, a_5, a_6, b_3, b_2, a_7, b_4, b_6, b_5, b_7\}, \  
\{a_2, a_3, a_4, a_5, a_6, b_3, b_2, b_5, b_6, b_4, a_7, b_7\}, \  
\{a_2, a_3, a_4, a_5, a_6, b_3, b_2, b_5, a_7, b_4, b_6, b_7\}, \\  
& \{a_2, a_3, a_4, a_5, a_6, b_3, b_2, a_7, b_5, b_4, b_6, b_7\}, \  
\{a_2, a_3, a_4, a_5, a_6, b_3, b_2, a_7, b_5, b_6, b_4, b_7\}, \  
\{a_2, a_3, a_4, a_5, b_4, a_6, b_2, b_3, a_7, b_5, b_6, b_7\}, \\  
& \{a_2, a_3, a_4, a_5, a_6, b_4, b_2, b_3, a_7, b_5, b_6, b_7\}, \  
\{a_2, a_3, a_4, a_5, a_6, a_7, b_2, b_3, b_4, b_5, b_6, b_7\}, \  
\{a_2, a_3, a_4, a_5, a_6, a_7, b_2, b_3, b_4, b_6, b_5, b_7\}, \\  
& \{a_2, a_3, a_4, a_5, b_5, a_6, b_2, b_3, a_7, b_4, b_6, b_7\}, \  
\{a_2, a_3, a_4, a_5, a_6, b_5, b_2, b_3, b_6, b_4, a_7, b_7\}, \  
\{a_2, a_3, a_4, a_5, a_6, b_5, b_2, b_3, a_7, b_4, b_6, b_7\}, \\  
& \{a_2, a_3, a_4, a_5, a_6, a_7, b_2, b_3, b_5, b_6, b_4, b_7\}, \  
\{a_2, a_3, a_4, a_5, a_6, a_7, b_2, b_3, b_6, b_4, b_5, b_7\}, \  
\{a_2, a_3, a_4, a_5, a_6, a_7, b_2, b_3, b_6, b_5, b_4, b_7\}, \\  
& \{a_2, a_3, a_4, a_5, b_4, a_6, b_2, a_7, b_3, b_5, b_6, b_7\}, \  
\{a_2, a_3, a_4, a_5, b_4, a_6, b_2, a_7, b_3, b_6, b_5, b_7\}, \  
\{a_2, a_3, a_4, a_5, a_6, b_4, b_2, a_7, b_3, b_5, b_6, b_7\}, \\  
& \{a_2, a_3, a_4, a_5, a_6, b_4, b_2, a_7, b_3, b_6, b_5, b_7\}, \  
\{a_2, a_3, a_4, a_5, a_6, a_7, b_2, b_4, b_3, b_5, b_6, b_7\}, \  
\{a_2, a_3, a_4, a_5, a_6, a_7, b_2, b_4, b_3, b_6, b_5, b_7\}, \\  
& \{a_2, a_3, a_4, a_5, a_6, b_5, b_2, a_7, b_3, b_4, b_6, b_7\}, \  
\{a_2, a_3, a_4, a_5, a_6, b_5, b_2, a_7, b_3, b_6, b_4, b_7\}, \  
\{a_2, a_3, a_4, a_5, a_6, a_7, b_2, b_5, b_3, b_4, b_6, b_7\}, \\  
& \{a_2, a_3, a_4, a_5, a_6, a_7, b_2, b_5, b_3, b_6, b_4, b_7\}, \  
\{a_2, a_3, a_4, a_5, a_6, a_7, b_2, b_6, b_3, b_4, b_5, b_7\}, \  
\{a_2, a_3, a_4, a_5, a_6, a_7, b_2, b_6, b_3, b_5, b_4, b_7\}, \\  
& \{a_2, a_3, a_4, a_5, a_6, a_7, b_2, b_7, b_3, b_4, b_5, b_6\}, \  
\{a_2, a_3, a_4, a_5, a_6, a_7, b_2, b_7, b_3, b_5, b_4, b_6\}, \  
\{a_2, a_3, a_4, a_5, b_4, a_6, b_2, b_5, a_7, b_3, b_6, b_7\}, \\  
& \{a_2, a_3, a_4, a_5, b_4, a_6, b_2, a_7, b_5, b_3, b_6, b_7\}, \  
\{a_2, a_3, a_4, a_5, a_6, b_4, b_2, b_5, a_7, b_3, b_6, b_7\}, \  
\{a_2, a_3, a_4, a_5, a_6, b_4, b_2, a_7, b_5, b_3, b_6, b_7\}, \\  
& \{a_2, a_3, a_4, a_5, a_6, a_7, b_2, b_4, b_5, b_3, b_6, b_7\}, \  
\{a_2, a_3, a_4, a_5, a_6, a_7, b_2, b_4, b_5, b_6, b_3, b_7\}, \  
\{a_2, a_3, a_4, a_5, a_6, a_7, b_2, b_4, b_6, b_3, b_5, b_7\}, \\  
& \{a_2, a_3, a_4, a_5, a_6, a_7, b_2, b_4, b_6, b_5, b_3, b_7\}, \  
\{a_2, a_3, a_4, a_5, a_6, a_7, b_2, b_5, b_4, b_6, b_3, b_7\}, \  
\{a_2, a_3, b_3, a_4, a_5, b_4, a_6, b_2, a_7, b_5, b_6, b_7\}, \\  
& \{a_2, a_3, b_3, a_4, a_5, a_6, a_7, b_2, b_4, b_5, b_6, b_7\}, \  
\{a_2, a_3, b_3, a_4, a_5, a_6, a_7, b_2, b_4, b_6, b_5, b_7\}, \  
\{a_2, a_3, a_4, b_3, b_4, a_5, a_6, b_2, a_7, b_5, b_6, b_7\}, \\  
& \{a_2, a_3, a_4, b_3, a_5, b_4, a_6, b_2, a_7, b_5, b_6, b_7\}, \  
\{a_2, a_3, a_4, b_3, a_5, a_6, b_4, b_2, a_7, b_5, b_6, b_7\}, \  
\{a_2, a_3, a_4, b_3, a_5, a_6, a_7, b_2, b_4, b_5, b_6, b_7\}, \\  
& \{a_2, a_3, a_4, b_3, a_5, a_6, a_7, b_2, b_4, b_6, b_5, b_7\}, \  
\{a_2, a_3, a_4, b_3, a_5, a_6, a_7, b_2, b_5, b_4, b_6, b_7\}, \  
\{a_2, a_3, a_4, b_3, a_5, a_6, a_7, b_2, b_5, b_6, b_4, b_7\}, \\  
& \{a_2, a_3, a_4, b_3, a_5, a_6, a_7, b_2, b_6, b_4, b_5, b_7\}, \  
\{a_2, a_3, a_4, b_3, a_5, a_6, a_7, b_2, b_6, b_5, b_4, b_7\}, \  
\{a_2, a_3, a_4, a_5, b_3, b_4, a_6, b_2, a_7, b_5, b_6, b_7\}, \\  
& \{a_2, a_3, a_4, a_5, b_3, a_6, b_4, b_2, a_7, b_5, b_6, b_7\}, \  
\{a_2, a_3, a_4, a_5, b_3, a_6, a_7, b_2, b_4, b_5, b_6, b_7\}, \  
\{a_2, a_3, a_4, a_5, b_3, a_6, a_7, b_2, b_4, b_6, b_5, b_7\}, \\  
& \{a_2, a_3, a_4, a_5, b_3, b_5, a_6, b_2, a_7, b_4, b_6, b_7\}, \  
\{a_2, a_3, a_4, a_5, b_3, a_6, a_7, b_2, b_5, b_4, b_6, b_7\}, \  
\{a_2, a_3, a_4, a_5, b_3, a_6, a_7, b_2, b_5, b_6, b_4, b_7\}, \\  
& \{a_2, a_3, a_4, a_5, b_3, a_6, a_7, b_2, b_6, b_4, b_5, b_7\}, \  
\{a_2, a_3, a_4, a_5, b_3, a_6, a_7, b_2, b_6, b_5, b_4, b_7\}, \  
\{a_2, a_3, a_4, a_5, b_4, b_3, a_6, b_2, a_7, b_5, b_6, b_7\}, \\  
& \{a_2, a_3, a_4, a_5, a_6, b_3, b_4, b_2, a_7, b_5, b_6, b_7\}, \  
\{a_2, a_3, a_4, a_5, a_6, b_3, a_7, b_2, b_4, b_5, b_6, b_7\}, \  
\{a_2, a_3, a_4, a_5, b_4, a_6, a_7, b_2, b_3, b_5, b_6, b_7\}, \\  
& \{a_2, a_3, a_4, a_5, b_4, a_6, a_7, b_2, b_3, b_6, b_5, b_7\}, \  
\{a_2, a_3, a_4, a_5, a_6, b_4, a_7, b_2, b_3, b_5, b_6, b_7\}, \  
\{a_2, a_3, b_3, a_4, a_5, b_4, a_6, a_7, b_2, b_6, b_5, b_7\}, \\  
& \{a_2, a_3, a_4, b_3, a_5, b_4, a_6, a_7, b_2, b_6, b_5, b_7\}, \  
\{a_2, a_3, a_4, b_3, a_5, a_6, b_4, a_7, b_2, b_5, b_6, b_7\}, \  
\{a_2, a_3, a_4, b_3, a_5, a_6, b_4, a_7, b_2, b_6, b_5, b_7\}, \\  
& \{a_2, a_3, a_4, b_3, a_5, a_6, b_5, a_7, b_2, b_4, b_6, b_7\} 
\end{align*}

%% file: dataB_2.txt
\begin{align*}
&\text{$470$--$592$}\\
&\{a_4, a_5, b_4, b_5, a_6, a_7, b_6, a_8, b_7, b_8, b_2\}, \
\{a_4, a_5, b_4, b_5, a_6, a_7, b_6, a_8, b_2, b_7, b_8\}, \
\{a_4, a_5, b_4, b_5, a_6, a_7, b_6, b_2, a_8, b_7, b_8\}, \\
&\{a_4, a_5, b_4, b_5, a_6, a_7, a_8, b_6, b_7, b_8, b_2\}, \
\{a_4, a_5, b_4, b_5, a_6, a_7, a_8, b_6, b_8, b_7, b_2\}, \
\{a_4, a_5, b_4, b_5, a_6, a_7, a_8, b_6, b_2, b_7, b_8\}, \\
&\{a_4, a_5, b_4, b_5, a_6, a_7, b_2, b_6, a_8, b_7, b_8\}, \
\{a_4, a_5, b_4, b_5, a_6, b_2, a_7, b_6, a_8, b_7, b_8\}, \
\{a_4, a_5, b_4, b_5, a_6, a_7, a_8, b_7, b_6, b_8, b_2\}, \\
&\{a_4, a_5, b_4, b_5, a_6, a_7, a_8, b_2, b_6, b_7, b_8\}, \
\{a_4, a_5, b_4, b_5, a_6, a_7, b_2, a_8, b_6, b_7, b_8\}, \
\{a_4, a_5, b_4, b_5, a_6, a_7, a_8, b_2, b_6, b_8, b_7\}, \\
&\{a_4, a_5, b_4, b_5, a_6, a_7, b_2, a_8, b_6, b_8, b_7\}, \
\{a_4, a_5, b_4, b_5, a_6, a_7, a_8, b_7, b_2, b_6, b_8\}, \
\{a_4, a_5, b_4, b_5, a_6, a_7, b_2, b_7, a_8, b_6, b_8\}, \\
&\{a_4, a_5, b_4, b_5, b_2, a_6, a_7, a_8, b_6, b_8, b_7\}, \
\{a_4, a_5, b_4, a_6, b_5, b_6, a_7, a_8, b_2, b_7, b_8\}, \
\{a_4, a_5, b_4, a_6, b_5, b_6, a_7, b_2, a_8, b_7, b_8\}, \\
&\{a_4, a_5, b_4, a_6, b_5, a_7, b_6, a_8, b_2, b_7, b_8\}, \
\{a_4, a_5, b_4, a_6, b_5, a_7, b_6, b_2, a_8, b_7, b_8\}, \
\{a_4, a_5, b_4, a_6, b_5, a_7, a_8, b_6, b_7, b_8, b_2\}, \\
&\{a_4, a_5, b_4, a_6, b_5, a_7, a_8, b_6, b_8, b_7, b_2\}, \
\{a_4, a_5, b_4, a_6, b_5, a_7, a_8, b_6, b_2, b_7, b_8\}, \
\{a_4, a_5, b_4, a_6, b_5, a_7, b_2, b_6, a_8, b_7, b_8\}, \\
&\{a_4, a_5, b_4, a_6, b_5, a_7, a_8, b_7, b_6, b_8, b_2\}, \
\{a_4, a_5, b_4, a_6, b_5, a_7, a_8, b_2, b_6, b_7, b_8\}, \
\{a_4, a_5, b_4, a_6, b_5, a_7, b_2, a_8, b_6, b_7, b_8\}, \\
&\{a_4, a_5, b_4, a_6, b_5, a_7, a_8, b_2, b_6, b_8, b_7\}, \
\{a_4, a_5, b_4, a_6, b_5, a_7, b_2, a_8, b_6, b_8, b_7\}, \
\{a_4, a_5, b_4, a_6, b_5, a_7, a_8, b_7, b_8, b_6, b_2\}, \\
&\{a_4, a_5, b_4, a_6, b_5, a_7, a_8, b_7, b_2, b_6, b_8\}, \
\{a_4, a_5, b_4, a_6, b_5, a_7, b_2, b_7, a_8, b_6, b_8\}, \
\{a_4, a_5, b_4, b_2, b_5, a_6, a_7, a_8, b_6, b_7, b_8\}, \\
&\{a_4, a_5, b_4, b_2, b_5, a_6, a_7, a_8, b_6, b_8, b_7\}, \
\{a_4, a_5, b_4, b_2, b_5, a_6, a_7, a_8, b_7, b_6, b_8\}, \
\{a_4, a_5, b_4, b_2, b_5, a_6, a_7, a_8, b_7, b_8, b_6\}, \\
&\{a_4, a_5, b_4, a_6, a_7, b_5, b_6, a_8, b_2, b_7, b_8\}, \
\{a_4, a_5, b_4, a_6, a_7, b_5, b_6, b_2, a_8, b_7, b_8\}, \
\{a_4, a_5, b_4, a_6, a_7, b_5, a_8, b_6, b_7, b_8, b_2\}, \\
&\{a_4, a_5, b_4, a_6, a_7, b_5, a_8, b_6, b_8, b_7, b_2\}, \
\{a_4, a_5, b_4, a_6, a_7, b_5, a_8, b_6, b_2, b_7, b_8\}, \
\{a_4, a_5, b_4, a_6, a_7, b_5, b_2, b_6, a_8, b_7, b_8\}, \\
&\{a_4, a_5, b_4, a_6, a_7, b_5, b_7, a_8, b_6, b_8, b_2\}, \
\{a_4, a_5, b_4, a_6, a_7, b_5, a_8, b_7, b_6, b_8, b_2\}, \
\{a_4, a_5, b_4, a_6, a_7, b_5, a_8, b_2, b_6, b_7, b_8\}, \\
&\{a_4, a_5, b_4, a_6, a_7, b_5, b_2, a_8, b_6, b_7, b_8\}, \
\{a_4, a_5, b_4, a_6, a_7, b_5, a_8, b_2, b_6, b_8, b_7\}, \
\{a_4, a_5, b_4, a_6, a_7, b_5, b_2, a_8, b_6, b_8, b_7\}, \\
&\{a_4, a_5, b_4, a_6, b_2, b_5, a_7, a_8, b_6, b_7, b_8\}, \
\{a_4, a_5, b_4, a_6, b_2, b_5, a_7, a_8, b_6, b_8, b_7\}, \
\{a_4, a_5, b_4, a_6, a_7, b_5, b_7, a_8, b_2, b_6, b_8\}, \\
&\{a_4, a_5, b_4, a_6, a_7, b_5, b_7, b_2, a_8, b_6, b_8\}, \
\{a_4, a_5, b_4, a_6, a_7, b_5, a_8, b_7, b_8, b_6, b_2\}, \
\{a_4, a_5, b_4, a_6, a_7, b_5, a_8, b_7, b_2, b_6, b_8\}, \\
&\{a_4, a_5, b_4, a_6, a_7, b_5, b_2, b_7, a_8, b_6, b_8\}, \
\{a_4, a_5, b_4, a_6, b_2, b_5, a_7, a_8, b_7, b_6, b_8\}, \
\{a_4, a_5, b_4, a_6, b_2, b_5, a_7, a_8, b_7, b_8, b_6\}, \\
&\{a_4, a_5, b_4, b_2, a_6, b_5, a_7, a_8, b_6, b_7, b_8\}, \
\{a_4, a_5, b_4, b_2, a_6, b_5, a_7, a_8, b_6, b_8, b_7\}, \
\{a_4, a_5, b_4, b_2, a_6, b_5, a_7, a_8, b_7, b_6, b_8\}, \\
&\{a_4, a_5, b_4, b_2, a_6, b_5, a_7, a_8, b_7, b_8, b_6\}, \
\{a_4, a_5, b_4, a_6, a_7, b_6, b_5, a_8, b_2, b_7, b_8\}, \
\{a_4, a_5, b_4, a_6, a_7, b_6, b_5, b_2, a_8, b_7, b_8\}, \\
&\{a_4, a_5, b_4, a_6, a_7, a_8, b_5, b_6, b_7, b_8, b_2\}, \
\{a_4, a_5, b_4, a_6, a_7, a_8, b_5, b_6, b_8, b_7, b_2\}, \
\{a_4, a_5, b_4, a_6, a_7, a_8, b_5, b_6, b_2, b_7, b_8\}, \\
&\{a_4, a_5, b_4, a_6, a_7, a_8, b_5, b_7, b_6, b_8, b_2\}, \
\{a_4, a_5, b_4, a_6, a_7, a_8, b_5, b_8, b_6, b_7, b_2\}, \
\{a_4, a_5, b_4, a_6, a_7, a_8, b_5, b_2, b_6, b_7, b_8\}, \\
&\{a_4, a_5, b_4, a_6, a_7, b_2, b_5, a_8, b_6, b_7, b_8\}, \
\{a_4, a_5, b_4, a_6, a_7, a_8, b_5, b_2, b_6, b_8, b_7\}, \
\{a_4, a_5, b_4, a_6, a_7, b_2, b_5, a_8, b_6, b_8, b_7\}, \\
&\{a_4, a_5, b_4, a_6, b_2, a_7, b_5, a_8, b_6, b_7, b_8\}, \
\{a_4, a_5, b_4, a_6, b_2, a_7, b_5, a_8, b_6, b_8, b_7\}, \
\{a_4, a_5, b_4, a_6, a_7, a_8, b_5, b_7, b_8, b_6, b_2\}, \\
&\{a_4, a_5, b_4, a_6, a_7, a_8, b_5, b_7, b_2, b_6, b_8\}, \
\{a_4, a_5, b_4, a_6, a_7, b_2, b_5, b_7, a_8, b_6, b_8\}, \
\{a_4, a_5, b_4, a_6, a_7, a_8, b_5, b_8, b_7, b_6, b_2\}, \\
&\{a_4, a_5, b_4, a_6, a_7, a_8, b_5, b_2, b_7, b_6, b_8\}, \
\{a_4, a_5, b_4, a_6, a_7, b_2, b_5, a_8, b_7, b_6, b_8\}, \
\{a_4, a_5, b_4, a_6, a_7, a_8, b_5, b_8, b_2, b_6, b_7\}, \\
&\{a_4, a_5, b_4, a_6, b_2, a_7, b_5, b_7, a_8, b_6, b_8\}, \
\{a_4, a_5, b_4, a_6, b_2, a_7, b_5, a_8, b_7, b_6, b_8\}, \
\{a_4, a_5, b_4, a_6, a_7, b_2, b_5, a_8, b_7, b_8, b_6\}, \\
&\{a_4, a_5, b_4, a_6, b_2, a_7, b_5, a_8, b_7, b_8, b_6\}, \
\{a_4, a_5, b_4, b_2, a_6, a_7, b_5, a_8, b_6, b_7, b_8\}, \
\{a_4, a_5, b_4, b_2, a_6, a_7, b_5, a_8, b_6, b_8, b_7\}, \\  
&\{a_4, a_5, b_4, b_2, a_6, a_7, b_5, a_8, b_7, b_6, b_8\}, \
\{a_4, a_5, b_4, b_2, a_6, a_7, b_5, a_8, b_7, b_8, b_6\}, \
\{a_4, a_5, b_4, a_6, a_7, b_6, a_8, b_5, b_7, b_8, b_2\}, \\
&\{a_4, a_5, b_4, a_6, a_7, b_6, a_8, b_5, b_8, b_7, b_2\}, \
\{a_4, a_5, b_4, a_6, a_7, b_6, a_8, b_5, b_2, b_7, b_8\}, \
\{a_4, a_5, b_4, a_6, a_7, a_8, b_6, b_5, b_7, b_8, b_2\}, \\
&\{a_4, a_5, b_4, a_6, a_7, a_8, b_6, b_5, b_2, b_7, b_8\}, \
\{a_4, a_5, b_4, a_6, a_7, a_8, b_7, b_5, b_6, b_8, b_2\}, \
\{a_4, a_5, b_4, a_6, a_7, a_8, b_2, b_5, b_6, b_7, b_8\}, \\
&\{a_4, a_5, b_4, a_6, a_7, b_2, a_8, b_5, b_6, b_7, b_8\}, \
\{a_4, a_5, b_4, a_6, a_7, a_8, b_2, b_5, b_6, b_8, b_7\}, \
\{a_4, a_5, b_4, a_6, a_7, b_2, a_8, b_5, b_6, b_8, b_7\}, \\
&\{a_4, a_5, b_4, a_6, b_2, a_7, a_8, b_5, b_6, b_7, b_8\}, \
\{a_4, a_5, b_4, a_6, b_2, a_7, a_8, b_5, b_6, b_8, b_7\}, \
\{a_4, a_5, b_4, a_6, a_7, a_8, b_7, b_5, b_8, b_6, b_2\}, \\
&\{a_4, a_5, b_4, a_6, a_7, a_8, b_7, b_5, b_2, b_6, b_8\}, \
\{a_4, a_5, b_4, a_6, a_7, a_8, b_2, b_5, b_7, b_6, b_8\}, \
\{a_4, a_5, b_4, a_6, a_7, b_2, a_8, b_5, b_7, b_6, b_8\}, \\
&\{a_4, a_5, b_4, a_6, a_7, a_8, b_2, b_5, b_8, b_6, b_7\}, \
\{a_4, a_5, b_4, a_6, a_7, b_2, a_8, b_5, b_8, b_6, b_7\}, \
\{a_4, a_5, b_4, a_6, b_2, a_7, a_8, b_5, b_7, b_6, b_8\}, \\
&\{a_4, a_5, b_4, a_6, b_2, a_7, a_8, b_5, b_8, b_6, b_7\}, \
\{a_4, a_5, b_4, a_6, a_7, a_8, b_2, b_5, b_7, b_8, b_6\}, \
\{a_4, a_5, b_4, a_6, a_7, b_2, a_8, b_5, b_7, b_8, b_6\}, \\
&\{a_4, a_5, b_4, a_6, a_7, a_8, b_2, b_5, b_8, b_7, b_6\}, \
\{a_4, a_5, b_4, a_6, a_7, b_2, a_8, b_5, b_8, b_7, b_6\}, \
\{a_4, a_5, b_4, a_6, b_2, a_7, a_8, b_5, b_7, b_8, b_6\}, \\
&\{a_4, a_5, b_4, a_6, b_2, a_7, a_8, b_5, b_8, b_7, b_6\}, \
\{a_4, a_5, b_4, b_2, a_6, a_7, a_8, b_5, b_6, b_7, b_8\}, \
\{a_4, a_5, b_4, b_2, a_6, a_7, a_8, b_5, b_6, b_8, b_7\}, \\
&\{a_4, a_5, b_4, b_2, a_6, a_7, a_8, b_5, b_8, b_6, b_7\}, \
\{a_4, a_5, b_4, b_2, a_6, a_7, a_8, b_5, b_7, b_8, b_6\}, \
\{a_4, a_5, b_4, b_2, a_6, a_7, a_8, b_5, b_8, b_7, b_6\}, \\
&\{a_4, a_5, b_4, a_6, a_7, b_6, a_8, b_7, b_5, b_8, b_2\}, \
\{a_4, a_5, b_4, a_6, a_7, b_6, a_8, b_2, b_5, b_7, b_8\}, \
\{a_4, a_5, b_4, a_6, a_7, b_6, b_2, a_8, b_5, b_7, b_8\}, 
\end{align*}

\begin{align*}
&\text{$593$--$648$}\\
&\{a_4, a_5, b_4, a_6, a_7, b_6, a_8, b_2, b_5, b_8, b_7\}, \
\{a_4, a_5, b_4, a_6, a_7, b_6, b_2, a_8, b_5, b_8, b_7\}, \
\{a_4, a_5, b_4, a_6, b_2, b_6, a_7, a_8, b_5, b_7, b_8\}, \\
&\{a_4, a_5, b_4, a_6, b_2, b_6, a_7, a_8, b_5, b_8, b_7\}, \
\{a_4, a_5, b_4, a_6, a_7, a_8, b_6, b_7, b_5, b_8, b_2\}, \
\{a_4, a_5, b_4, a_6, a_7, a_8, b_6, b_8, b_5, b_7, b_2\}, \\
&\{a_4, a_5, b_4, a_6, a_7, a_8, b_6, b_2, b_5, b_7, b_8\}, \
\{a_4, a_5, b_4, a_6, a_7, b_2, b_6, a_8, b_5, b_7, b_8\}, \
\{a_4, a_5, b_4, a_6, a_7, a_8, b_6, b_2, b_5, b_8, b_7\}, \\
&\{a_4, a_5, b_4, a_6, a_7, b_2, b_6, a_8, b_5, b_8, b_7\}, \
\{a_4, a_5, b_4, a_6, b_2, a_7, b_6, a_8, b_5, b_7, b_8\}, \
\{a_4, a_5, b_4, a_6, b_2, a_7, b_6, a_8, b_5, b_8, b_7\}, \\
&\{a_4, a_5, b_4, a_6, a_7, a_8, b_7, b_6, b_5, b_8, b_2\}, \
\{a_4, a_5, b_4, a_6, a_7, a_8, b_7, b_8, b_5, b_6, b_2\}, \
\{a_4, a_5, b_4, a_6, a_7, a_8, b_7, b_2, b_5, b_6, b_8\}, \\
&\{a_4, a_5, b_4, a_6, a_7, a_8, b_7, b_2, b_5, b_8, b_6\}, \
\{a_4, a_5, b_4, a_6, a_7, b_2, b_7, a_8, b_5, b_8, b_6\}, \
\{a_4, a_5, b_4, a_6, a_7, b_6, b_7, a_8, b_2, b_5, b_8\}, \\
&\{a_4, a_5, b_4, a_6, a_7, b_6, a_8, b_7, b_8, b_5, b_2\}, \
\{a_4, a_5, b_4, a_6, a_7, b_6, a_8, b_7, b_2, b_5, b_8\}, \
\{a_4, a_5, b_4, a_6, a_7, a_8, b_6, b_7, b_8, b_5, b_2\}, \\
&\{a_4, a_5, b_4, a_6, a_7, a_8, b_6, b_7, b_2, b_5, b_8\}, \
\{a_4, a_5, b_4, a_6, a_7, a_8, b_6, b_8, b_7, b_5, b_2\}, \
\{a_4, a_5, b_4, a_6, a_7, a_8, b_6, b_8, b_2, b_5, b_7\}, \\
&\{a_4, a_5, b_4, a_6, a_7, a_8, b_7, b_6, b_2, b_5, b_8\}, \
\{a_4, a_5, b_4, a_6, a_7, a_8, b_7, b_8, b_6, b_5, b_2\}, \
\{a_4, a_5, a_6, b_4, b_5, b_6, a_7, a_8, b_2, b_7, b_8\}, \\
&\{a_4, a_5, a_6, b_4, b_5, a_7, b_6, a_8, b_2, b_7, b_8\}, \
\{a_4, a_5, a_6, b_4, b_5, a_7, a_8, b_6, b_8, b_7, b_2\}, \
\{a_4, a_5, a_6, b_4, b_5, a_7, a_8, b_6, b_2, b_7, b_8\}, \\
&\{a_4, a_5, a_6, b_4, b_5, a_7, a_8, b_7, b_6, b_8, b_2\}, \
\{a_4, a_5, a_6, b_4, b_5, a_7, a_8, b_2, b_6, b_7, b_8\}, \
\{a_4, a_5, a_6, b_4, b_5, a_7, b_2, a_8, b_6, b_7, b_8\}, \\
&\{a_4, a_5, a_6, b_4, b_5, a_7, a_8, b_2, b_6, b_8, b_7\}, \
\{a_4, a_5, a_6, b_4, b_5, a_7, a_8, b_7, b_8, b_6, b_2\}, \
\{a_4, a_5, b_2, b_4, b_5, a_6, a_7, a_8, b_7, b_6, b_8\}, \\
&\{a_4, a_5, a_6, b_4, a_7, b_5, b_6, a_8, b_2, b_7, b_8\}, \
\{a_4, a_5, a_6, b_4, a_7, b_5, a_8, b_6, b_8, b_7, b_2\}, \
\{a_4, a_5, a_6, b_4, a_7, b_5, a_8, b_6, b_2, b_7, b_8\}, \\
&\{a_4, a_5, a_6, b_4, a_7, b_5, a_8, b_7, b_6, b_8, b_2\}, \
\{a_4, a_5, a_6, b_4, a_7, b_5, a_8, b_2, b_6, b_7, b_8\}, \
\{a_4, a_5, a_6, b_4, a_7, b_5, b_2, a_8, b_6, b_7, b_8\}, \\
&\{a_4, a_5, a_6, b_4, a_7, b_5, a_8, b_2, b_6, b_8, b_7\}, \
\{a_4, a_5, a_6, b_4, a_7, b_5, b_2, a_8, b_6, b_8, b_7\}, \
\{a_4, a_5, a_6, b_4, a_7, b_5, a_8, b_7, b_8, b_6, b_2\}, \\
&\{a_4, a_5, a_6, b_4, a_7, a_8, b_5, b_6, b_7, b_8, b_2\}, \
\{a_4, a_5, a_6, b_4, a_7, a_8, b_5, b_6, b_8, b_7, b_2\}, \
\{a_4, a_5, a_6, b_4, a_7, a_8, b_5, b_6, b_2, b_7, b_8\}, \\
&\{a_4, a_5, a_6, b_4, a_7, a_8, b_5, b_2, b_6, b_7, b_8\}, \
\{a_4, a_5, a_6, b_4, a_7, a_8, b_5, b_7, b_8, b_6, b_2\}, \
\{a_4, a_5, a_6, b_4, a_7, a_8, b_5, b_7, b_2, b_6, b_8\}, \\
&\{a_4, a_5, a_6, b_4, a_7, a_8, b_5, b_8, b_7, b_6, b_2\}, \
\{a_4, a_5, a_6, b_4, a_7, a_8, b_2, b_5, b_6, b_7, b_8\}, \
\{a_4, a_5, a_6, b_4, a_7, b_2, a_8, b_5, b_6, b_7, b_8\}, \\
&\{a_4, a_5, b_2, b_4, a_6, a_7, a_8, b_5, b_6, b_7, b_8\}, \
\{a_4, a_5, a_6, b_5, a_7, a_8, b_6, b_7, b_8, b_4, b_2\} 
\end{align*}

%% file: dataC.txt
\begin{align*}
& C=2160 C_{1
}-2616 C_{2
}-180 C_{3
}-384 C_{4
}-1956 C_{5
}-528 C_{6
}+1596 C_{7
}-2100 C_{8} \\
& +1524 C_{9
}-240 C_{10
}-1812 C_{11
}-1020 C_{12
}+132 C_{13
}+1344 C_{14
}+1584 C_{15
}-2520 C_{16} \\
& +2400 C_{17
}-156 C_{18
}-1224 C_{19
}-1020 C_{20
}-4236 C_{21
}-1164 C_{22
}-12 C_{23
}-4140 C_{24} \\
& -1224 C_{25
}-7140 C_{26
}-1032 C_{27
}-7428 C_{28
}-840 C_{29
}-2736 C_{30
}-3228 C_{31
}-936 C_{32} \\
& -2736 C_{33
}-4680 C_{34
}+496 C_{35
}+524 C_{36
}+3780 C_{37
}+4028 C_{38
}-788 C_{39
}+1264 C_{40} \\
& -3384 C_{41
}-1076 C_{42
}+516 C_{43
}+372 C_{44
}-1664 C_{45
}+3000 C_{46
}+1980 C_{47
}+756 C_{48} \\
& +4672 C_{49
}-5768 C_{50
}-3204 C_{51
}+3072 C_{52
}-1540 C_{53
}-6316 C_{54
}-6268 C_{55
}-9808 C_{56} \\
& -3476 C_{57
}-7832 C_{58
}-8640 C_{59
}-13092 C_{60
}+4516 C_{61
}+5060 C_{62
}+8092 C_{63
}-848 C_{64} \\
& +2976 C_{65
}+1992 C_{66
}+588 C_{67
}+772 C_{68
}+1088 C_{69
}+244 C_{70
}+5824 C_{71
}+4860 C_{72} \\
& +2136 C_{73
}+2280 C_{74
}-1296 C_{75
}+664 C_{76
}-3260 C_{77
}-4436 C_{78
}+916 C_{79
}-7724 C_{80} \\
& -2296 C_{81
}+6256 C_{82
}+9744 C_{83
}-1780 C_{84
}+4808 C_{85
}+4796 C_{86
}-700 C_{87
}+1284 C_{88} \\
& +3784 C_{89
}+5580 C_{90
}+7172 C_{91
}+4832 C_{92
}+4832 C_{93
}+5168 C_{94
}+84 C_{95
}+564 C_{96} \\
& +2088 C_{97
}+1980 C_{98
}-804 C_{99
}-1756 C_{100
}-692 C_{101
}+2928 C_{102
}-2496 C_{103
}-1344 C_{104} \\
& +516 C_{105
}+2532 C_{106
}-78876 C_{107
}+3684 C_{108
}-234400 C_{109
}+3612 C_{110
}-49828 C_{111
}-4452 C_{112} \\
& -122892 C_{113
}+696 C_{114
}-122436 C_{115
}-10032 C_{116
}+10844 C_{117
}-26378 C_{118
}-20362 C_{119
}+924 C_{120} \\
& -37936 C_{121
}+2172 C_{122
}-179514 C_{123
}+792 C_{124
}+3792 C_{125
}-409286 C_{126
}+1740 C_{127
}-215036 C_{128} \\
& -9924 C_{129
}-499622 C_{130
}+20316 C_{131
}-1920 C_{132
}+27168 C_{133
}-553120 C_{134
}+1812 C_{135
}-8440 C_{136} \\
& -3996 C_{137
}-11222 C_{138
}-9432 C_{139
}-114290 C_{140
}+960 C_{141
}-15018 C_{142
}-4944 C_{143
}-349734 C_{144} \\
& -8724 C_{145
}-196928 C_{146
}-9636 C_{147
}-368506 C_{148
}+16476 C_{149
}+22116 C_{150
}+14304 C_{151
}-624962 C_{152} \\
& +7548 C_{153
}+81792 C_{154
}-1716 C_{155
}+19188 C_{156
}-124728 C_{157
}+20304 C_{158
}+8712 C_{159
}+8948 C_{160} \\
& -7248 C_{161
}+7080 C_{162
}-495912 C_{163
}+2964 C_{164
}+3060 C_{165
}+14928 C_{166
}+204 C_{167
}-392446 C_{168} \\
& -11388 C_{169
}-23126 C_{170
}+2532 C_{171
}-97280 C_{172
}+4656 C_{173
}-201028 C_{174
}-43634 C_{175
}+1056 C_{176} \\
& -83688 C_{177
}-5844 C_{178
}-188572 C_{179
}+1752 C_{180
}-112556 C_{181
}+44288 C_{182
}-10416 C_{183
}+5100 C_{184} \\
& +288 C_{185
}+288 C_{186
}-153500 C_{187
}-188918 C_{188
}+8916 C_{189
}-19486 C_{190
}+17424 C_{191
}-119176 C_{192} \\
& +8340 C_{193
}+18612 C_{194
}-390946 C_{195
}+17580 C_{196
}-138656 C_{197
}+13176 C_{198
}-388138 C_{199
}+43860 C_{200} \\
& +6216 C_{201
}+42864 C_{202
}-412368 C_{203
}+25356 C_{204
}+34612 C_{205
}+4500 C_{206
}-25844 C_{207
}-6192 C_{208} \\
& -154118 C_{209
}-10992 C_{210
}-29418 C_{211
}-15360 C_{212
}-297790 C_{213
}+11472 C_{214
}-21748 C_{215
}+13368 C_{216} \\
& -145026 C_{217
}+13428 C_{218
}+28258 C_{219
}-8664 C_{220
}+130460 C_{221
}-25164 C_{222
}-428252 C_{223
}-20232 C_{224} \\
& +440158 C_{225
}-204 C_{226
}+18192 C_{227
}-118308 C_{228
}+4812 C_{229
}-17472 C_{230
}-51790 C_{231
}-1344 C_{232} \\
& +4752 C_{233
}+15336 C_{234
}-325702 C_{235
}-5652 C_{236
}-18264 C_{237
}-39024 C_{238
}-266390 C_{239
}-35304 C_{240} \\
& +233132 C_{241
}-12042 C_{242
}-5244 C_{243
}-45756 C_{244
}+3612 C_{245
}-244776 C_{246
}+18230 C_{247
}+8712 C_{248} \\
& +1932 C_{249
}-62988 C_{250
}-8436 C_{251
}-5952 C_{252
}-150596 C_{253
}-3864 C_{254
}-2136 C_{255
}-125024 C_{256} \\
& +1992 C_{257
}+1728 C_{258
}-38492 C_{259
}-194014 C_{260
}-2460 C_{261
}+28500 C_{262
}-4728 C_{263
}-37756 C_{264} \\
& +4536 C_{265
}-155092 C_{266
}+1836 C_{267
}-325192 C_{268
}-116450 C_{269
}+2244 C_{270
}-223856 C_{271
}+876 C_{272} \\
& +1260 C_{273
}-3264 C_{274
}-350994 C_{275
}-8196 C_{276
}+74408 C_{277
}-187996 C_{278
}+2772 C_{279
}-5054 C_{280} 
\end{align*}

\begin{align*}
& +1260 C_{281
}-97248 C_{282
}-3732 C_{283
}-183678 C_{284
}-7920 C_{285
}-271274 C_{286
}+8160 C_{287
}+42640 C_{288} \\
& +19236 C_{289
}-4118 C_{290
}+16740 C_{291
}-180 C_{292
}-16500 C_{293
}-299976 C_{294
}-16776 C_{295
}+431810 C_{296} \\
& -15576 C_{297
}-13080 C_{298
}-256796 C_{299
}+1248 C_{300
}-12984 C_{301
}-209848 C_{302
}-27468 C_{303
}-31896 C_{304} \\
& -352402 C_{305
}-34152 C_{306
}+3024 C_{307
}-492 C_{308
}-20448 C_{309
}-366528 C_{310
}-22128 C_{311
}-183650 C_{312} \\
& +13392 C_{313
}+11052 C_{314
}+15660 C_{315
}+11340 C_{316
}-11700 C_{317
}-9540 C_{318
}-12744 C_{319
}-14664 C_{320} \\
& -16692 C_{321
}-288 C_{322
}-144 C_{323
}-1572 C_{324
}-4860 C_{325
}-158448 C_{326
}-2976 C_{327
}-10128 C_{328} \\
& -60114 C_{329
}-1152 C_{330
}-9816 C_{331
}-479106 C_{332
}-20484 C_{333
}-3156 C_{334
}-11148 C_{335
}-20892 C_{336} \\
& -258324 C_{337
}-34692 C_{338
}-144 C_{339
}-2940 C_{340
}+4956 C_{341
}-219758 C_{342
}-5244 C_{343
}-27420 C_{344} \\
& -58968 C_{345
}+3144 C_{346
}+3900 C_{347
}-372174 C_{348
}-780 C_{349
}+3528 C_{350
}-8352 C_{351
}-28008 C_{352} \\
& -209924 C_{353
}-30576 C_{354
}-19200 C_{355
}-18516 C_{356
}-208418 C_{357
}+2784 C_{358
}+312 C_{359
}-18420 C_{360} \\
& -207992 C_{361
}+2892 C_{362
}-31080 C_{363
}-37392 C_{364
}-145046 C_{365
}-49656 C_{366
}+3624 C_{367
}-5040 C_{368} \\
& -29568 C_{369
}-88052 C_{370
}-37632 C_{371
}-78438 C_{372
}-187038 C_{373
}+160810 C_{374
}+6990 C_{375
}-37872 C_{376} \\
& -43702 C_{377
}+182330 C_{378
}+56778 C_{379
}-11448 C_{380
}-250466 C_{381
}-22364 C_{382
}-121402 C_{383
}+43126 C_{384} \\
& -4968 C_{385
}-2758 C_{386
}+265038 C_{387
}+231260 C_{388
}+288100 C_{389
}+3240 C_{390
}-18324 C_{391
}+194794 C_{392} \\
& +283548 C_{393
}+515637 C_{394
}-2316 C_{395
}+323536 C_{396
}+481808 C_{397
}-51474 C_{398
}-27764 C_{399
}+720 C_{400} \\
& -73830 C_{401
}+184948 C_{402
}+219794 C_{403
}-6576 C_{404
}-125438 C_{405
}+547099 C_{406
}+373157 C_{407
}+71732 C_{408} \\
& +165186 C_{409
}-28358 C_{410
}+294458 C_{411
}+10008 C_{412
}+5328 C_{413
}-25298 C_{414
}+89932 C_{415
}-121304 C_{416} \\
& +322458 C_{417
}-35738 C_{418
}+202833 C_{419
}-31276 C_{420
}+726792 C_{421
}+9972 C_{422
}-2760 C_{423
}-94436 C_{424} \\
& +537527 C_{425
}-94412 C_{426
}+767108 C_{427
}+8400 C_{428
}-6624 C_{429
}+274923 C_{430
}+499389 C_{431
}-47262 C_{432} \\
& +184238 C_{433
}+5520 C_{434
}+204 C_{435
}+8414 C_{436
}+77216 C_{437
}-100388 C_{438
}+229320 C_{439
}-5532 C_{440} \\
& +392222 C_{441
}-1212 C_{442
}+10068 C_{443
}-6012 C_{444
}-169735 C_{445
}-79048 C_{446
}+346939 C_{447
}-85488 C_{448} \\
& +584211 C_{449
}+1464 C_{450
}-4620 C_{451
}-24846 C_{452
}+237574 C_{453
}-46686 C_{454
}-197994 C_{455
}+121016 C_{456} \\
& +65142 C_{457
}-100316 C_{458
}-98460 C_{459
}+40844 C_{460
}-456 C_{461
}-5976 C_{462
}-85866 C_{463
}-14984 C_{464} \\
& -62636 C_{465
}-120454 C_{466
}-16498 C_{467
}+5076 C_{468
}+2868 C_{469
}-25334 C_{470
}+27912 C_{471
}+145006 C_{472} \\
& +135304 C_{473
}-546915 C_{474
}-2172 C_{475
}-45760 C_{476
}+1080 C_{477
}-4764 C_{478
}-67712 C_{479
}+72894 C_{480} \\
& -48172 C_{481
}+106442 C_{482
}-26852 C_{483
}+6792 C_{484
}-564 C_{485
}-54540 C_{486
}+65725 C_{487
}-78148 C_{488} \\
& -45105 C_{489
}-56430 C_{490
}-196792 C_{491
}+6420 C_{492
}+4344 C_{493
}+16046 C_{494
}+3506 C_{495
}+518556 C_{496} \\
& +14436 C_{497
}+8448 C_{498
}-202292 C_{499
}-120572 C_{500
}-106874 C_{501
}-167120 C_{502
}-457738 C_{503
}-144 C_{504} \\
& -43668 C_{505
}+180350 C_{506
}+648 C_{507
}-4176 C_{508
}+8546 C_{509
}+49202 C_{510
}-54142 C_{511
}+290692 C_{512} \\
& +481507 C_{513
}+7296 C_{514
}+528 C_{515
}+556337 C_{516
}+675051 C_{517
}+11832 C_{518
}+5772 C_{519
}+236561 C_{520} \\
& +554071 C_{521
}-28810 C_{522
}-88202 C_{523
}+6432 C_{524
}+3732 C_{525
}+19202 C_{526
}+53332 C_{527
}+81882 C_{528} \\
& +118434 C_{529
}+18048 C_{530
}+7872 C_{531
}-71070 C_{532
}-152514 C_{533
}-837964 C_{534
}+14904 C_{535
}+12516 C_{536} \\
& -48792 C_{537
}+139934 C_{538
}+333706 C_{539
}+195122 C_{540
}-5550 C_{541
}+2520 C_{542
}-9334 C_{543
}+12558 C_{544} \\
& +293278 C_{545
}-9900 C_{546
}-34994 C_{547
}+192956 C_{548
}-25094 C_{549
}+589686 C_{550
}+2120 C_{551
}-15360 C_{552} \\
& -145954 C_{553
}-54654 C_{554
}-49338 C_{555
}-46614 C_{556
}-16744 C_{557
}+705212 C_{558
}-240 C_{559
}-136382 C_{560} \\
& +66116 C_{561
}-9728 C_{562
}-4672 C_{563
}-11484 C_{564
}-23148 C_{565
}+14558 C_{566
}+10302 C_{567
}+239250 C_{568} \\
& -15864 C_{569
}+417133 C_{570
}-22538 C_{571
}+77869 C_{572
}-5844 C_{573
}-171993 C_{574
}+201203 C_{575
}+8368 C_{576} \\
& +60594 C_{577
}+99130 C_{578
}-48346 C_{579
}-52546 C_{580
}-273972 C_{581
}-230570 C_{582
}-905942 C_{583
}-214963 C_{584} \\
& +217439 C_{585
}+145604 C_{586
}-12708 C_{587
}+306690 C_{588
}+10198 C_{589
}+137228 C_{590
}+346058 C_{591
}+186122 C_{592} \\
& +158174 C_{593
}+31594 C_{594
}+157286 C_{595
}+270234 C_{596
}+295582 C_{597
}-78596 C_{598
}+123409 C_{599
}-671038 C_{600} 
\end{align*}

\begin{align*}
& +29136 C_{601
}+28520 C_{602
}+33346 C_{603
}-17632 C_{604
}+225875 C_{605
}+48128 C_{606
}-79590 C_{607
}+65096 C_{608} \\
& -94964 C_{609
}+149694 C_{610
}+486422 C_{611
}+506574 C_{612
}-85370 C_{613
}+480920 C_{614
}+3720 C_{615
}+798026 C_{616} \\
& +201043 C_{617
}+313962 C_{618
}+292500 C_{619
}-11268 C_{620
}+107479 C_{621
}+590932 C_{622
}-171482 C_{623
}+967553 C_{624} \\
& -61746 C_{625
}+783790 C_{626
}+65073 C_{627
}+111508 C_{628
}+1202622 C_{629
}+30466 C_{630
}-163486 C_{631
}+1039771 C_{632} \\
& +30178 C_{633
}+811337 C_{634
}+138379 C_{635
}-37184 C_{636
}-35378 C_{637
}+55370 C_{638
}-168724 C_{639
}-429016 C_{640} \\
& -79884 C_{641
}-155454 C_{642
}+7584 C_{643
}-231631 C_{644
}-485137 C_{645
}-39500 C_{646
}-658713 C_{647}
\end{align*}